\documentclass[12pt]{article}

\usepackage{amsmath}
\usepackage{amssymb}
\usepackage{amscd}
\usepackage{amsthm}
\usepackage{indentfirst}
\usepackage[english,frenchb]{babel}

\theoremstyle{plain}
\newtheorem{thm}{Th\'eor\`eme}[section]
\newtheorem{lem}[thm]{Lemme}
\newtheorem{cor}[thm]{Corollaire}
\newtheorem{prop}[thm]{Proposition}
\newtheorem{propdfn}[thm]{Proposition-d\'efinition}
\theoremstyle{definition}
\newtheorem{dfn}[thm]{D\'efinition}

\newtheorem{rmq}[thm]{Remarque}
\newtheorem{exm}[thm]{Exemple}

\DeclareMathOperator{\Spec}{Spec}
\DeclareMathOperator{\pic}{Pic}

\newcommand{\Sa}{S}
\newcommand{\s}{\rho}
\newcommand{\A}{\mathcal{A}}
\newcommand{\fppf}{\text{\rm fppf}}
\newcommand{\ket}{\text{\rm k\'et}}
\newcommand{\et}{\text{\rm \'et}}
\newcommand{\kpl}{\text{\rm kpl}}
\newcommand{\topo}{\text{\rm top}}

\newcommand{\B}{\mathcal{B}}

\newcommand{\gm}{\mathbf{G}_{{\rm m}}}
\newcommand{\gmlog}{\mathbf{G}_{{\rm m,log}}}

\newcommand{\gmy}{\mathbf{G}_{{\rm m},Z}}
\newcommand{\gmv}{\mathbf{G}_{{\rm m},V}}

\newcommand{\cl}{cl}
\DeclareMathOperator{\Cl}{Cl}

\DeclareMathOperator{\homr}{Hom}
\DeclareMathOperator{\Hom}{\underline{Hom}}
\DeclareMathOperator{\ext}{Ext}
\DeclareMathOperator{\Ext}{\underline{Ext}}

\DeclareMathOperator{\coker}{Coker}
\DeclareMathOperator{\Rac}{Rac}
\DeclareMathOperator{\RAC}{RAC}

\DeclareMathOperator{\Raclog}{C^{\rm log}}
\DeclareMathOperator{\Sym}{Sym}

\DeclareMathOperator{\diviseur}{div}
\DeclareMathOperator{\DivRat}{DivRat}
\DeclareMathOperator{\Divp}{DivPrinc}

\begin{document}

\title{Cohomologie log plate, actions mod\'er\'ees et structures galoisiennes}

\author{Jean Gillibert}

\date{octobre 2010}

\maketitle

\begin{abstract}
Soit $X$ un log sch\'ema fin et satur\'e, et soit $G$ un sch\'ema en groupes commutatif fini et plat sur le sch\'ema sous-jacent
\`a $X$. Si l'on peut consid\'erer que les $G$-torseurs pour la topologie fppf sont des objets non ramifi\'es par nature, au
contraire les $G$-torseurs pour la topologie log plate nous permettent d'envisager de la ramification mod\'er\'ee. En nous
servant des travaux de Kato, nous d\'efinissons un concept de structure galoisienne pour ces torseurs, puis g\'en\'eralisons
les constructions pr\'ec\'edentes de l'auteur (homomorphisme de classes pour les vari\'et\'es ab\'eliennes \`a r\'eduction semi-stable)
dans ce cadre, levant au passage certaines restrictions.

\medskip

\begin{otherlanguage}{english}
\begin{center}
\textbf{Abstract}
\end{center}

Let $X$ be a fine and saturated log scheme, and let $G$ be a commutative finite flat group scheme over the underlying scheme of $X$. If $G$-torsors for the fppf topology can be thought of as being unramified objects by nature, then $G$-torsors for the log flat topology allow us to consider tame ramification. Using the results of Kato, we define a concept of Galois structure for these torsors, then we generalize the author's previous constructions (class-invariant homomorphism for semi-stable abelian varieties) in this new setting, thus dropping some restrictions.
\end{otherlanguage}
\end{abstract}



\section{Introduction}

Le probl\`eme qui motive ce travail est la construction (dans la continuit\'e de \cite{gil1} et \cite{gil2}) du \emph{class-invariant homomorphism} (introduit en premier par M. J. Taylor dans \cite{t2}) pour les vari\'et\'es ab\'eliennes \`a mauvaise r\'eduction. Il s'agissait \`a l'origine de comprendre pourquoi la composante neutre du mod\`ele de N\'eron (et, plus g\'en\'eralement, l'accouplement de monodromie) jouait un r\^ole si particulier dans les travaux pr\'ec\'edents de l'auteur, et comment on pouvait \'etendre la construction des invariants de classes \`a tous les points de la vari\'et\'e, sans aucune limitation impos\'ee par les groupes de composantes. Nous fournissons ici une r\'eponse compl\`ete \`a ce probl\`eme (cf. th\'eor\`emes \ref{thintro1} et \ref{thintro2}).

Au cours de nos investigations, le langage des log sch\'emas s'est av\'er\'e \^etre un cadre naturel et agr\'eable pour faire na\^itre et formuler correctement notre r\'eponse.

Une fois l'outil logarithmique adopt\'e, trois questions \'etroitements imbriqu\'ees se sont d\'egag\'ees : le probl\`eme g\'en\'eral de prolongement des torseurs fppf classiques par des torseurs log plats, le cas particulier des torseurs obtenus gr\^ace au cobord d'une isog\'enie entre sch\'emas ab\'eliens, et enfin l'\'etude de la structure galoisienne des torseurs log plats.
La premi\`ere question r\'ev\`ele des liens \'etroits entre torseurs log plats et actions mod\'er\'ees, la deuxi\`eme se r\'esout en prolongeant des biextensions pour la topologie log plate, et la troisi\`eme admet deux r\'eponses : l'une classique, l'autre  logarithmique.

Au final, cet article se veut \^etre une premi\`ere pierre dans un nouveau champ d'investigation : l'utilisation des outils logarithmiques (log sch\'emas et topologie log plate) pour l'\'etude des actions mod\'er\'ees de sch\'emas en groupes et des structures galoisiennes qui s'en d\'eduisent.

\subsection{Prolongement de torseurs}
\label{hunhun}

Soient $X$ un sch\'ema n\oe{}th\'erien int\`egre r\'egulier, $G$ un $X$-sch\'ema en groupes commutatif, fini localement libre, et $U$ un ouvert de $X$ dont le compl\'ementaire est un diviseur $D$. Alors le morphisme de restriction
$$
H^1_{\fppf}(X,G)\longrightarrow H^1_{\fppf}(U,G)
$$
est injectif (cf. proposition \ref{injectrestr}). On aimerait comprendre son image. En d'autres termes, \'etant donn\'e un $G$-torseur sur $U$, on aimerait savoir sous quelles conditions celui-ci admet un prolongement en un $G$-torseur sur $X$.

Ce probl\`eme est de nature locale. Soit $x\in D$ un point de codimension $1$, comme $X$ est r\'egulier on sait que $S:=\Spec(\mathcal{O}_{X,x})$ est un trait. Supposons que $G\times_X S$ soit un $S$-sch\'ema en groupes \'etale --- c'est le cas en particulier si la caract\'eristique r\'esiduelle de $x$ est premi\`ere \`a l'ordre de $G$ --- alors, pour qu'un $G$-torseur sur $U$ se prolonge au-dessus du point $x$, il faut et il suffit que l'extension correspondante du point g\'en\'erique de $S$ soit non ramifi\'ee (c'est-\`a-dire que le morphisme de traits associ\'e soit \'etale).

Une telle situation n'est pas fr\'equente. Par exemple, \'etant donn\'ee une extension galoisienne de corps de nombres, il n'y a aucune raison pour qu'elle soit non ramifi\'ee en dehors des places divisant l'ordre de son groupe de Galois. Notons toutefois que la ramification est toujours mod\'er\'ee en une place ne divisant pas l'ordre du groupe.

En d'autres termes, \'etant donn\'e un $G$-torseur sur $U$, le fait d'\^etre prolongeable par un $G$-torseur sur $X$ (pour la topologie fppf) est une condition trop restrictive si l'on souhaite comprendre en finesse la ramification sur le compl\'ementaire de $U$.
Remplacer la topologie fppf par une topologie plus fine (mais n\'eanmoins sous-canonique) ne cr\'ee pas de nouveaux torseurs repr\'esentables.
Face \`a cette situation, il convient de quitter le monde des torseurs, ou de se placer dans une autre cat\'egorie que celle des sch\'emas.

Dans cet article, nous avons choisi d'\'etudier les torseurs dans la cat\'egorie des log sch\'emas fins et satur\'es, munie de la topologie log plate d\'efinie par Kato. Pour fixer les id\'ees, regardons de plus pr\`es le cas de figure pr\'ec\'edent. Supposons que $D$ soit un diviseur \`a croisements normaux.
Nous munissons $X$ de la structure logarithmique d\'efinie par $U$, laquelle est fine et satur\'ee (voir l'exemple \ref{debase}). Nous nous int\'eressons \`a pr\'esent au morphisme de restriction
$$
H^1_{\kpl}(X,G)\longrightarrow H^1_{\fppf}(U,G)
$$
o\`u le groupe de cohomologie de gauche est calcul\'e pour la topologie Kummer plate. Dans le cas pr\'esent, on montre que tout \'element de $H^1_{\text{\rm kpl}}(X,G)$  est repr\'esentable par un log sch\'ema (fin et satur\'e) dont le sch\'ema sous-jacent est fini localement libre sur $X$ (cf. proposition \ref{loclibre}). En outre, l'action de $G$ sur ce dernier est mod\'er\'ee au sens de Chinburg, Erez, Pappas et Taylor dans \cite{cept} (cf. th\'eor\`eme \ref{CEPTtame}).

Cette approche entretient des liens de parent\'e \'evidents avec celle de Grothendieck et Murre dans \cite{gm}, dont elle s'inspire (voir la remarque \ref{GMlike}).

La notion d'action mod\'er\'ee d\'efinie dans \cite{cept} m\'erite quelques pr\'ecisions. Dans ce langage, il s'entend que c'est bien l'action qui est mod\'er\'ee, pas la ramification de l'objet sur lequel on agit : l'extension sous-jacente \`a une action mod\'er\'ee peut tr\`es bien \^etre sauvagement ramifi\'ee. Un point central de cette th\'eorie est que toute action d'un sch\'ema en groupes fini de type multiplicatif est mod\'er\'ee. Plus g\'en\'eralement, toute action mod\'er\'ee d'un sch\'ema en groupes commutatif fini localement libre est induite, localement pour la topologie fppf, par une action d'un sch\'ema en groupes fini de type multiplicatif (voir \cite[Theorem 6.4]{cept}).

La premi\`ere cons\'equence tangible de notre choix logarithmique est la suivante : les actions mod\'er\'ees \'etudi\'ees par les auteurs pr\'ec\'edents, qui avaient jusqu'ici une existence purement g\'eom\'etrique, entrent de plain-pied dans le champ de l'alg\`ebre homologique, et s'enrichissent de toute la souplesse de cette derni\`ere.

\subsection{Vari\'et\'es ab\'eliennes et isog\'enies}
\label{hundeux}

A pr\'esent $S$ est un sch\'ema de Dedekind (\emph{i.e.} n\oe{}th\'erien r\'egulier de dimension $1$) connexe de point g\'en\'erique $\eta=\Spec(K)$.
Soit $\A$ (resp. $\A^t$) le mod\`ele de N\'eron d'une $K$-vari\'et\'e ab\'elienne (resp. de sa vari\'et\'e duale), on note $U\subseteq S$ l'ouvert de bonne r\'eduction de $\A$, et $D$ le diviseur compl\'ementaire de $U$. Nous munissons $S$ de la log structure d\'efinie par l'ouvert $U$.

On note $W_U$ la biextension de Weil de $(\A_U,\A^t_U)$ par $\gm$ (pour la topologie fppf), laquelle traduit la dualit\'e entre les sch\'emas ab\'eliens $\A_U$ et $\A^t_U$. Son $\gm$-torseur sous-jacent est le fibr\'e de Poincar\'e.

On note $\A^0$ le $S$-sch\'ema obtenu en recollant la fibre g\'en\'erique de $\A$ avec les composantes neutres de ses fibres sp\'eciales. On note $\Phi$ le faisceau quotient $\A/\A^0$  que l'on appelle, par abus de langage, le groupe des composantes de $\A$.
On note $\Phi'$ le groupe des composantes de $\A^t$. On appelle (\'egalement par abus) accouplement de monodromie l'objet obtenu en recollant, pour chaque $s\in D$, l'accouplement  d\'efini par \cite[expos\'e IX, 1.2.1]{gro7} (apr\`es changement de base $\Spec(\mathcal{O}_{S,s})\rightarrow S$).

Si $\Gamma\subseteq\Phi$ est un sous-groupe ouvert de $\Phi$, on note $\A^{\Gamma}$ l'image r\'eciproque de $\Gamma$ par le morphisme naturel $\A\rightarrow \Phi$. Il s'agit d'un sous-groupe ouvert de $\A$.

Supposons que $\Gamma\subseteq\Phi$ et $\Gamma'\subseteq\Phi'$ soient des sous-groupes ouverts de $\Phi$ et $\Phi'$. Si $\Gamma$ et $\Gamma'$ sont orthogonaux sous l'accouplement de monodromie, alors il existe une unique biextension de $(\A^{\Gamma},\A^{t,\Gamma'})$ par $\gm$ prolongeant $W_U$, que nous noterons  $W^{\Gamma,\Gamma'}$.

Soit $G\subseteq \A^{\Gamma}$ un sous-groupe fini et plat de $\A$. Soit $\phi_U:\A_U\rightarrow \B_U$ l'isog\'enie obtenue en quotientant (pour la topologie fppf) $\A_U$ par $G_U$, et soit $\phi_U^t: \B_U^t\rightarrow \A_U^t$ l'isog\'enie duale, de noyau $G_U^D$. On dispose alors d'une suite exacte (pour la topologie fppf sur $X$)
\begin{equation}
\label{schabel}
\begin{CD}
0 @>>> G_U^D @>>> \B_U^t @>\phi_U^t>> \A_U^t @>>> 0 \\
\end{CD}
\end{equation}
d'o\`u, en appliquant le foncteur des sections globales, un morphisme cobord
$$
\begin{CD}
\Delta_U:\A_U^t(U) @>>> H^1_{\fppf}(U,G^D). \\
\end{CD}
$$

On souhaite \`a pr\'esent travailler sur $S$ tout entier, sans \^etre oblig\'e de se restreindre \`a l'ouvert $U$ de bonne r\'eduction. Autrement dit, on se demande pour quels points $y\in \A_U^t(U)=\A^t(S)$ le torseur $\Delta_U(y)$ se prolonge en un $G^D$-torseur sur $S$. Cette question s'inscrit dans la probl\'ematique du paragraphe \ref{hunhun}.

Tout d'abord, on cherche \`a d'\'etendre, de fa\c{c}on suffisamment naturelle, la suite exacte $(\ref{schabel})$ de faisceaux sur $U$ en une suite exacte de faisceaux sur $S$. Dans \cite{gil1} et \cite{gil2} nous avons construit, en travaillant dans un petit site fppf convenable, une suite exacte
\begin{equation}
\label{schpabel}
\begin{CD}
0 @>>> G^D @>>> \Ext^1(B,\gm) @>\phi^*>> \Ext^1(\A^{\Gamma},\gm) @>>> 0 \\
\end{CD}
\end{equation}
dont les termes sont des faisceaux fppf (\emph{a priori} non repr\'esentables). Ici, $\phi:\A^{\Gamma}\rightarrow B$ est le quotient de $\A^{\Gamma}$ par $G$ pour la topologie fppf (en g\'en\'eral $B$ n'est pas un sous-groupe d'un mod\`ele de N\'eron). On en d\'eduit des morphismes
$$
\begin{CD}
\A^{t,\Gamma'}(S) @>\gamma>> \ext^1_{\fppf}(\A^{\Gamma},\gm) @>\delta>> H^1_{\fppf}(S,G^D) \\
\end{CD}
$$
o\`u $\delta$ est le cobord de la suite $(\ref{schpabel})$ et $\gamma$ est le morphisme associ\'e \`a la biextension $W^{\Gamma,\Gamma'}$ d\'efinie plus haut.
Ainsi, le cobord $\Delta_U$ \`a valeurs dans $H^1_{\fppf}(U,G^D)$ se rel\`eve en un morphisme \`a valeurs dans $H^1_{\fppf}(S,G^D)$, \`a condition de se restreindre aux points de $\A^{t,\Gamma'}(S)$. Comme en g\'en\'eral ce dernier est un sous-groupe strict de $\A^t(S)$, on se demande s'il est possible d'\'etendre $\delta\circ\gamma$ \`a $\A^t(S)$ tout entier. Comme nous l'avons dit dans le paragraphe \ref{hunhun}, l'obstruction \`a cela est un probl\`eme de ramification, lequel s'explicite ici de la fa\c{c}on suivante : si $y$ est un point de $\A_{\eta}^t(K)$ alors l'extension engendr\'ee par les points de $\phi^t$-division de $y$ a tendance \`a \^etre ramifi\'ee en les points de $S$ au-dessus desquels $y$ ne se r\'eduit pas en un point de la composante neutre de $\A^t$.

En suivant toujours le cheminement du paragraphe \ref{hunhun}, notre nouvel objectif est de relever $\Delta_U$ en un morphisme \`a valeurs dans $H^1_{\kpl}(S,G^D)$. Pour cela, il faut d'abord prolonger la biextension $W_U$ en une biextension de $(\A,\A^t)$ par $\gm$ pour la topologie Kummer log plate. Nous l'avons fait dans l'article \cite{gil5}. Apr\`es quelques ajustements techniques (cf. paragraphe \ref{dernier}), nous sommes en mesure de d\'efinir des morphismes
$$
\begin{CD}
\A^t(S) @>\gamma^{\rm log}>> \ext^1_{\kpl}(\A,\gm) @>\delta^{\rm log}>> H^1_{\kpl}(S,G^D) \\
\end{CD}
$$
tous les groupes \'etant calcul\'es \`a l'aide de cohomologie Kummer log plate. Nous r\'esumons nos r\'esultats dans le th\'eor\`eme ci-dessous (cf. paragraphe \ref{dernier}).

\begin{thm}
\label{thintro1}
\begin{enumerate}
\item[$(i)$] Le cobord $\Delta_U$ de la suite \eqref{schabel} se rel\`eve en un morphisme
$$
\begin{CD}
\Delta:\A^t(S) @>>> H^1_{\kpl}(S,G^D)
\end{CD}
$$
\item[$(ii)$] Pour tout point $y\in\A^t(S)$, $\Delta(y)$ appartient au sous-groupe $H^1_{\fppf}(S,G^D)$ si et seulement si l'image de $y$ dans $\Phi'$ est orthogonale \`a l'image de $G$ dans $\Phi$ sous l'accouplement de monodromie.
\end{enumerate}
\end{thm}

Le point $(i)$ constitue une propri\'et\'e frappante des torseurs obtenus en divisant des points par une isog\'enie entre vari\'et\'es ab\'eliennes : ceux-ci sont en effet de nature mod\'er\'ee au-dessus des places de mauvaise r\'eduction.
Le point $(ii)$ montre que la ramification desdits torseurs est contr\^ol\'ee par l'accouplement de monodromie. En particulier les constructions pr\'ec\'edentes \'etaient optimales du point de vue de la cohomologie fppf.

\subsection{Structures galoisiennes}

Comme nous l'avons dit plus haut, l'une de nos motivations est l'\'etude du \emph{class-invariant homomorphism} (introduit par M. J. Taylor dans \cite{t2}) pour les vari\'et\'es ab\'eliennes \`a mauvaise r\'eduction.

Rappelons tout d'abord que l'on doit \`a Waterhouse (voir \cite[Theorem 5]{w}) la construction d'un homomorphisme mesurant la structure galoisienne des $G^D$-torseurs
$$
\begin{CD}
\pi:H^1(X,G^D) @>>> \pic(G)\\
\end{CD}
$$
dont l'interpr\'etation (en termes alg\'ebriques) est \emph{grosso modo} la suivante : l'action de $G^D$ sur un torseur $T$ induit sur l'alg\`ebre de $T$ une structure de comodule sur l'alg\`ebre de $G^D$ ;  l'alg\`ebre de $T$ est cons\'equemment munie d'une structure de module sur l'alg\`ebre de $G$, lequel module s'av\`ere \^etre localement libre de rang $1$, donnant ainsi naissance \`a une classe $\pi(T):=\mathcal{O}_T\otimes_{\mathcal{O}_G} (\mathcal{O}_{G^D})^{-1}$ dans le groupe $\pic(G)$. Notons que $\pi$ admet une autre d\'efinition, de nature cohomologique.

Ces deux d\'efinitions donnent naissance \`a deux g\'en\'eralisations naturelles de la construction de Waterhouse dans le cadre de la topologie Kummer log plate.

La premi\`ere (cf. d\'efinition \ref{mapcl}) est une application
$$
\begin{CD}
\cl:H^1_{\kpl}(X,G^D) @>>> \pic(G) \\
\end{CD}
$$
qui envoie un $G^D$-torseur $Y$ sur la classe de $\mathcal{O}_Y\otimes_{\mathcal{O}_G} (\mathcal{O}_{G^D})^{-1}$ dans $\pic(G)$. Pour que ceci soit bien d\'efini, il faut que l'alg\`ebre sous-jacente \`a $Y$ soit localement libre de rang $1$ sur l'alg\`ebre de $G$, ce qui est le cas sous des hypoth\`eses raisonnables (voir le corollaire \ref{CHtame}). Notons ici que $\cl$ n'est pas un morphisme de groupes en g\'en\'eral (cf. exemple \ref{CDNconstant}).

La deuxi\`eme (cf. d\'efinition \ref{loghomo}) est de nature cohomologique : nous d\'efinissons un morphisme
$$
\begin{CD}
\pi^{\rm log}:H^1_{\kpl}(X,G^D) @>>> H^1_{\kpl}(G,\gm)\\
\end{CD}
$$
qui mesure la structure galoisienne logarithmique des $G^D$-torseurs pour la cohomologie Kummer plate. Ce nouvel invariant nous semble pertinent sous deux points de vue. Tout d'abord, il combine la structure de module galoisien d'un torseur et des donn\'ees de ramification (cf. exemples \ref{logramiftrait} et \ref{quelonveut}). Ensuite, c'est un morphisme de groupes, contrairement \`a l'invariant $\cl$ qui ne mesure que la structure de module. Nous projetons d'approfondir l'\'etude de cet invariant dans nos futurs travaux.

Il d\'ecoule du th\'eor\`eme \ref{thintro1} et de nos constructions deux g\'en\'eralisations possibles du \emph{class-invariant homomorphism}, selon que l'on mesure la structure galoisienne \`a l'aide de $\cl$ ou de $\pi^{\rm log}$. On peut r\'esumer ces constructions dans le th\'eor\`eme ci-dessous.

\begin{thm}
\label{thintro2}
Avec les notations et hypoth\`eses du paragraphe \ref{hundeux}, il existe une application
$$
\begin{CD}
\psi^{\cl}:\A^t(S) @>>> H^1_{\kpl}(S,G^D) @>\cl >> \pic(G) \\
\end{CD}
$$
et un morphisme de groupes
$$
\begin{CD}
\psi^{\rm log}:\A^t(S) @>>> H^1_{\kpl}(S,G^D) @>\pi^{\rm log}>> H^1_{\kpl}(G,\gm)\\
\end{CD}
$$
qui mesurent respectivement les structures galoisiennes classique et logarithmique des torseurs (pour la topologie Kummer plate) obtenus en divisant les points de $\A^t(S)$ par l'isog\'enie $\phi^t$. Quand on restreint $\psi^{\cl}$ et $\psi^{\rm log}$ au sous-groupe $\A^{t,0}(S)$ on retrouve le morphisme $\psi:\A^{t,0}(S) \rightarrow \pic(G)$ \'etudi\'e dans \cite{gil1} et \cite{gil2}.
\end{thm}

Supposons que $\A_{\eta}$ soit une courbe elliptique \`a r\'eduction semi-stable, et que l'ordre de $G$ soit premier \`a $6$. Alors (voir \cite{gil1} et \cite{gil2}) le morphisme $\psi:\A^{t,0}(S) \rightarrow \pic(G)$ s'annule sur les points de torsion --- ce qui g\'en\'eralise un r\'esultat d\'emontr\'e par Taylor, Srivastav, Agboola et Pappas dans le cas de bonne r\'eduction.
Cependant, sous les m\^emes hypoth\`eses, on ne peut pas esp\'erer que $\psi^{\rm log}$ ou $\psi^{\cl}$ s'annule sur les points de torsion qui tombent en dehors de la composante neutre.

Supposons $S$ affine. Dans le langage des alg\`ebres de Hopf, nos r\'esultats se traduisent de la fa\c{c}on suivante : soit $H$ l'alg\`ebre de Hopf de $G$, alors \`a tout point $y\in \A^t(S)$ on peut associer l'alg\`ebre $L(y)$ du torseur $\Delta(y)$, laquelle est munie d'une structure de $H^*$-objet mod\'er\'e au sens de Childs et Hurley (notons que $L(y)$ est l'ordre associ\'e \`a $H^*$ dans l'alg\`ebre du torseur $\Delta_U(y)$). De plus, l'alg\`ebre $L(y)$ est un $H^*$-objet galoisien si et seulement si l'image de $y$ dans $\Phi'$ est orthogonale \`a l'image de $G$ dans $\Phi$ sous l'accouplement de monodromie. Notons enfin que $\psi^{\cl}(y)$ est nul si et seulement si $L(y)\otimes_H (H^*)^{-1}$ est libre en tant que $H$-module.

R\'esumons bri\`evement le contenu de cet article. Dans la section \ref{deux}, nous rappelons les d\'efinitions de base concernant la cat\'egorie des log sch\'emas et les topologies de Grothendieck sur celle-ci (topologies Kummer log \'etale et Kummer log plate).

Dans la section \ref{trois}, nous \'etudions les propri\'et\'es des torseurs, pour la topologie Kummer log plate, sous un sch\'ema en groupes commutatif fini localement libre. Nous donnons une interpr\'etation des rev\^etements cycliques uniformes de Arsie et Vistoli \cite{av} en termes de $\mu_n$-torseurs. Puis nous montrons le th\'eor\`eme \ref{CEPTtame} qui relie les torseurs log plats aux actions mod\'er\'ees au sens de \cite{cept}.

Dans la section \ref{quatre}, nous d\'efinissons le morphisme $\pi^{\rm log}$ de structure galoisienne logarithmique, ainsi que l'application $\cl$ mesurant la structure galoisienne classique. Nous explicitons ces deux objets dans le cas particulier des $\mu_n$-torseurs. Le dernier paragraphe est consacr\'e \`a la d\'emonstration des th\'eor\`emes \ref{thintro1} et \ref{thintro2}.


\section{Log sch\'emas et topologies}
\label{deux}


\subsection{D\'efinitions et notations}

Nous rappelons ici succintement les d\'efinitions de base concernant les log sch\'emas. Pour plus de d\'etails, nous renvoyons le lecteur \`a l'article (de survol) d'Illusie \cite{illusie}, ainsi qu'\`a celui de Kato \cite{kato1}.

Tous les mono\"ides consid\'er\'es ici sont commutatifs unitaires. Le groupe des fractions totales d'un mono\"ide $P$ est not\'e $P^{\rm gp}$. On dit qu'un mono\"ide $P$ est int\`egre si le morphisme canonique $P\rightarrow P^{\rm gp}$ est injectif.  On dit qu'un mono\"ide $P$ est fin s'il est int\`egre et de type fini. On dit qu'un mono\"ide $P$ est satur\'e s'il est int\`egre et satisfait la condition suivante : pour tout $a\in P^{\rm gp}$, s'il existe un entier $n\geq 1$ tel que $a^n\in P$, alors $a\in P$. Un mono\"ide fs est un mono\"ide fin et satur\'e.

Une pr\'e-log structure sur un sch\'ema $X$ est la donn\'ee d'un faisceau de mono\"ides $M$ sur le petit site \'etale de $X$ et d'un morphisme $\alpha:M\rightarrow \mathcal{O}_X$ de faisceaux de mono\"ides, $\mathcal{O}_X$ \'etant ici muni de sa loi de multiplication. Une log structure est une pr\'e-log structure telle que $\alpha$ induise un isomorphisme $\alpha^{-1}(\mathcal{O}_X^*)\rightarrow \mathcal{O}_X^*$. \'Etant donn\'e une pr\'e-log structure, il existe une log structure satisfaisant la propri\'et\'e universelle qui s'entend, on l'appelle la log structure associ\'ee. Un log sch\'ema est un couple $(X,\alpha:M\rightarrow \mathcal{O}_X)$ o\`u $X$ est un sch\'ema et $\alpha:M\rightarrow \mathcal{O}_X$ est une log structure sur $X$. La notion de morphisme de log sch\'emas se d\'efinit de fa\c{c}on naturelle.

Le faisceau de mono\"ides d'un log sch\'ema $X$ est usuellement not\'e $M_X$. Le faisceau $\mathcal{O}_X^*$ est vu comme un sous-faisceau de $M_X$ gr\^ace au morphisme $\alpha$, le faisceau quotient $M_X/\mathcal{O}_X^*$ est alors not\'e $\overline{M}_X$.

Si $X$ est un sch\'ema, l'inclusion $\mathcal{O}_X^*\subset \mathcal{O}_X$  d\'efinit une log structure sur $X$, que l'on appelle la log structure triviale. La cat\'egorie des sch\'emas s'identifie ainsi \`a une sous-cat\'egorie pleine de celle des log sch\'emas. Plus \'esot\'eriquement,  ce foncteur d'inclusion $X\mapsto (X,\mathcal{O}_X^*\subset \mathcal{O}_X)$ est l'adjoint \`a droite du foncteur d'oubli $(Y,M_Y\rightarrow \mathcal{O}_X)\mapsto Y$ de la cat\'egorie des log  sch\'emas dans celle des sch\'emas.

Si $f:Y\rightarrow X$ est un morphisme de log sch\'emas, on notera $f^{\circ}$ le morphisme sur les sch\'emas sous-jacents (\emph{i.e.} l'image de $f$ par le foncteur d'oubli). On dit qu'un morphisme de log sch\'emas $f:Y\rightarrow X$ est strict si la log structure de $Y$ est l'image r\'eciproque de la log structure de $X$ par le morphisme $f^{\circ}$.

Si $X$ est un log sch\'ema, alors le plus grand ouvert de Zariski de $X$ sur lequel la log structure induite est triviale s'appelle l'ouvert de trivialit\'e de $X$.

Nous dirons d'un log sch\'ema qu'il est n\oe{}th\'erien (resp. localement n\oe{}th\'erien) si son sch\'ema sous-jacent poss\`ede cette propri\'et\'e.

\begin{exm}
Soit $P$ un mono\"ide, et soit $\mathbb{Z}[P]$ la $\mathbb{Z}$-alg\`ebre librement engendr\'e par $P$. Alors $\Spec(\mathbb{Z}[P])$ est muni d'une log structure (dite canonique) associ\'ee \`a la pr\'e-log structure d\'efinie par l'inclusion $P\subset \mathbb{Z}[P]$. L'ouvert de trivialit\'e est $\Spec(\mathbb{Z}[P^{\rm gp}])$.
\end{exm}

Une carte d'un log sch\'ema $X$ est la donn\'ee d'un mono\"ide $P$ et d'un morphisme $P\rightarrow M_X$ induisant un isomorphisme sur les log structures associ\'ees (par abus de langage, on note encore $P$ le faisceau constant sur $X$ de valeur $P$). On dit que la carte $P\rightarrow M_X$ est model\'ee sur $P$. De fa\c{c}on \'equivalente, une carte de $X$ est la donn\'ee d'un mono\"ide $P$ et d'un morphisme strict de log sch\'emas $X\rightarrow \Spec(\mathbb{Z}[P])$.

On dit qu'une log structure sur un sch\'ema $X$ est coh\'erente si, localement pour la topologie \'etale sur $X$, celle-ci admet une carte model\'ee sur un mono\"ide de type fini.

On dit qu'un log sch\'ema $X$ est fin (resp. fin et satur\'e) si sa log structure est coh\'erente, et si $M_X$ est un faisceau de mono\"ides int\`egres (resp. satur\'es). Un log sch\'ema fs est un log sch\'ema fin et satur\'e.

Soit $f:Y\rightarrow X$ un morphisme de log sch\'emas. Une carte de $f$ est la donn\'ee de cartes $P\rightarrow M_X$ et $Q\rightarrow M_Y$ pour $X$ et $Y$ respectivement, et d'un morphisme $u:P\rightarrow Q$ faisant commuter le diagramme
$$
\begin{CD}
P @>u>> Q \\
@VVV @VVV \\
f^{-1}M_X @>>> M_Y \\
\end{CD}
$$

Localement pour la topologie \'etale, tout morphisme entre deux log sch\'emas fins admet une carte.

On peut aussi exprimer la notion de carte d'un morphisme en termes de morphismes stricts. Soient $X\rightarrow \Spec(\mathbb{Z}[P])$ et $Y\rightarrow \Spec(\mathbb{Z}[Q])$ des cartes, et soit $u:P\rightarrow Q$ un morphisme compatible avec celles-ci. Alors $u$ est une carte de $f$ si et seulement si le morphisme $Y\rightarrow X\times_{\Spec(\mathbb{Z}[P])} \Spec(\mathbb{Z}[Q])$ est strict.

Dans la suite de cet article, tous les log sch\'emas consid\'er\'es seront fs. Toutes les op\'erations sur les log sch\'emas seront effectu\'ees dans la cat\'egorie des log sch\'emas fs.

\begin{exm}
\label{debase}
Soient $X$ un sch\'ema n\oe{}th\'erien r\'egulier, et $j:U\rightarrow X$ un ouvert dont le compl\'ementaire est un diviseur \`a croisements normaux sur $X$. Alors l'inclusion
$$
\mathcal{O}_X\cap j_*\mathcal{O}_U^*\longrightarrow \mathcal{O}_X
$$
induit une log structure fine et satur\'ee sur $X$. On dit qu'elle est d\'efinie par $U$, lequel est l'ouvert de trivialit\'e de cette log structure. Pour de plus amples d\'etails, on consultera \cite[paragraphe 2.1]{gil5}.
\end{exm}


\subsection{Topologies log \'etale et log plate}

Il s'agit d'abord de d\'efinir la log platitude, la log lissit\'e, et la log \'etalet\'e d'un morphisme de log sch\'emas. Nous reprenons les d\'efinitions de Kato.

\begin{dfn}
\label{logdefinitions}
Soit $f:Y\rightarrow X$ un morphisme de log sch\'emas fs. On dit que $f$ est log plat (resp. log lisse, resp. log \'etale) si localement pour la topologie fppf (resp. \'etale, resp. \'etale) sur $X$ et $Y$, $f$ admet une carte $u:P\rightarrow Q$ telle que
\begin{enumerate}
\item[$(i)$] Le morphisme $u^{\rm gp}:P^{\rm gp}\rightarrow Q^{\rm gp}$ est injectif (resp. est injectif et la partie de torsion de son conoyau est finie d'ordre inversible sur $X$, resp. est injectif et son conoyau est fini d'ordre inversible sur $X$) ;
\item[$(ii)$] Le morphisme strict $Y\rightarrow X\times_{\Spec(\mathbb{Z}[P])} \Spec(\mathbb{Z}[Q])$  est plat (resp. lisse, resp. \'etale) au sens classique.
\end{enumerate}
\end{dfn}

Nous allons introduire la notion de morphisme kumm\'erien, ce qui a pour effet principal d'imposer la finitude du conoyau de $u^{\rm gp}$ dans la d\'efinition ci-dessus.

\begin{dfn}
Soit $u:P\rightarrow Q$ un morphisme de mono\"ides fins et satur\'es. On dit que $u$ est kumm\'erien (ou de Kummer) s'il est injectif, et si pour tout $a\in Q$, il existe un entier $n\geq 1$ tel que $a^n\in u(P)$.
\end{dfn}

\begin{rmq}
Soit $u:P\rightarrow Q$ un morphisme injectif de mono\"ides fins et satur\'es. Alors les assertions suivantes sont \'equivalentes :
\begin{enumerate}
\item[(a)] $u$ est kumm\'erien ;
\item[(b)] $u$ est exact et $\coker(u^{\rm gp})$ est un groupe fini.
\end{enumerate}
(on dit que le morphisme $u$ est exact si $P=(u^{\rm gp})^{-1}(Q)$ dans $P^{\rm gp}$).
\end{rmq}

\begin{dfn}
Soit $f:Y\rightarrow X$ un morphisme de log sch\'emas fins et satur\'es. On dit que $f$ est kumm\'erien si, pour tout point $y\in Y$, le morphisme induit $\overline{M}_{X,f(y)}\rightarrow \overline{M}_{Y,y}$ est kumm\'erien.
\end{dfn}

Nous introduisons \`a pr\'esent les topologies que nous utiliserons dans la cat\'egorie des log sch\'emas fins et satur\'es.
Nous nous appuyons sur les travaux de Kato \cite{kato2} et Nizio\l{} \cite{niziol} pour les d\'efinitions et propri\'et\'es de ces topologies.

\begin{dfn}
On d\'efinit deux topologies de Grothendieck dans la cat\'egorie des log sch\'emas fins et satur\'es, de la fa\c{c}on suivante.
\begin{enumerate}
\item[$(i)$] La topologie \emph{Kummer log \'etale} est la topologie la moins fine telle que les familles ensemblistement surjectives de morphismes log \'etales kumm\'eriens soient couvrantes.
\item[$(ii)$] La topologie \emph{Kummer log plate} est la topologie la moins fine telle que les familles ensemblistement surjectives de morphismes log plats, kumm\'eriens et localement de pr\'esentation finie sur les sch\'emas sous-jacents, soient couvrantes.
\end{enumerate}
\end{dfn}

Pour \'epargner au lecteur de trop nombreuses r\'ep\'etitions, nous omettrons tant\^ot le log, tant\^ot le Kummer dans notre usage quotidien de la terminologie ci-dessus.

Un premier point important : les pr\'efaisceaux repr\'esentables (par des log sch\'emas) sont des faisceaux, ce qui s'exprime de la fa\c{c}on suivante.

\begin{thm}
La topologie Kummer \'etale est moins fine que la topologie Kummer plate, laquelle est moins fine que la topologie canonique.
\end{thm}

\begin{proof}
Il est clair que la topologie Kummer \'etale est moins fine que la topologie Kummer plate. Le fait que cette derni\`ere soit moins fine que la topologie canonique est d\^u \`a Kato (voir \cite[Theorem 3.1]{kato2} ou \cite[Theorem 2.20]{niziol}).
\end{proof}


\subsection{Torseurs kumm\'eriens standards}
\label{pardeuxtrois}

Nous allons voir \`a pr\'esent un proc\'ed\'e de construction de torseurs (pour la topologie log plate) sous des sch\'emas en groupes diagonalisables. La preuve est la m\^eme qu'en topologie Kummer \'etale (voir \cite[Proposition 3.2]{illusie}).

Nous adoptons temporairement la convention suivante : si $M$ est un mono\"ide fs, nous noterons $D(M)$ le sch\'ema $\Spec(\mathbb{Z}[M])$ muni de sa structure canonique de log sch\'ema fs. Dans le cas o\`u $M$ est un groupe ab\'elien, $D(M)$ est le sch\'ema en groupes diagonalisable associ\'e \`a $M$, muni de la log structure triviale.

\begin{propdfn}
\label{kummertorsor}
Soit $X$ un log sch\'ema muni d'une carte $P\rightarrow M_X$, et soit $u:P\rightarrow Q$ un morphisme kumm\'erien de mono\"ides fins et satur\'es. Alors la projection naturelle
$$
Y:=X\times_{\Spec(\mathbb{Z}[P])} \Spec(\mathbb{Z}[Q])\longrightarrow X
$$
est un torseur pour la topologie Kummer plate sous le groupe $D(Q^{\rm gp}/u(P^{\rm gp}))$. Un torseur obtenu de la sorte s'appelle un torseur kumm\'erien standard.
\end{propdfn}

\begin{proof}
D'apr\`es \cite[Remark 3.4, (c)]{illusie}, sous les hypoth\`eses envisag\'ees, le morphisme $D(Q)\rightarrow D(P)$ est formellement principal homog\`ene pour le groupe $D(Q^{\rm gp}/u(P^{\rm gp}))$. En d'autres termes, nous avons un isomorphisme
$$
D(Q^{\rm gp}/u(P^{\rm gp}))\times D(Q)\rightarrow D(Q)\times_{D(P)} D(Q),\quad (g,x)\mapsto (x, g.x)
$$
dans la cat\'egorie des log sch\'emas fs. Donc, par changement de base $X\rightarrow D(P)$, le morphisme $Y\rightarrow X$ est formellement principal homog\`ene de groupe $D(Q^{\rm gp}/u(P^{\rm gp}))$. D'autre part, il d\'ecoule ais\'ement des d\'efinitions que le morphisme $Y\rightarrow X$ est un \'epimorphisme pour la topologie Kummer plate, donc est un torseur.
\end{proof}

\begin{lem}
\label{niziolstyle}
Soit $f:Y\rightarrow X$ un morphisme de log sch\'emas, et soit $\alpha:P\rightarrow M_X$ une carte de $X$. On suppose que $f$ est log plat (resp. log lisse, resp. log \'etale).
\begin{enumerate}
\item[$(1)$] Localement pour la topologie fppf (resp. \'etale, resp. \'etale) sur $X$ et $Y$, $f$ admet une carte $u:P\rightarrow Q$ prolongeant $\alpha$ et satisfaisant les propri\'et\'es de la d\'efinition \ref{logdefinitions}.
\item[$(2)$] Si $f$ est kumm\'erien, alors $u$ l'est \'egalement.
\item[$(3)$] Si $P$ est sans torsion, alors on peut choisir $Q$ sans torsion.
\end{enumerate}
\end{lem}

\begin{proof}
Voir \cite[lemma 2.8]{niziol} pour le premier point. Le second point est clair, et le troisi\`eme est d\'emontr\'e par Nakayama dans \cite[Proposition A.2]{nakayama} (ce dernier travaille avec un morphisme log lisse, mais son raisonnement reste valide dans le cas log plat).
\end{proof}

On en d\'eduit le r\'esultat suivant, dont l'analogue en topologie Kummer \'etale est bien connu (voir \cite[Corollary 3.6]{illusie}).

\begin{cor}
\label{vidal}
Soit $X$ un log sch\'ema. La topologie Kummer plate sur $X$ est engendr\'ee par les recouvrements fppf classiques et les torseurs kumm\'eriens standards.
\end{cor}


\subsection{Comparaison avec la cohomologie classique}

Si $X$ est un sch\'ema (resp. un log sch\'ema), nous noterons $(Sch/X)$ (resp. $(fs/X)$) la cat\'egorie des sch\'emas (resp. des log sch\'emas fs) sur $X$. Nous utiliserons les symboles $\et$ et $\fppf$ pour d\'esigner les topologies \'etale et fppf classiques, ainsi que $\ket$ et $\kpl$ pour d\'esigner les topologies Kummer log \'etale et Kummer log plate.
Si $C$ est l'une de ces cat\'egories et si $\topo$ est l'une de ces topologies, nous noterons $C_{\topo}$ le site obtenu en munissant $C$ de la topologie $\topo$.
Le foncteur d'oubli induit un morphisme de sites
$$
\varepsilon :(fs/X)_{\kpl}\longrightarrow (Sch/X)_{\fppf}
$$

Soit $Y\rightarrow X$ un $X$-sch\'ema, et soit $h_Y$ le faisceau qu'il repr\'esente (\emph{i.e.} son foncteur des points); alors le faisceau $\varepsilon^*h_Y$ est repr\'esentable par le log sch\'ema obtenu en munissant $Y$ de la log structure image r\'eciproque de celle de $X$.

Si $G$ est un $X$-sch\'ema en groupes, nous noterons \'egalement (par abus de langage) $G$ le log sch\'ema en groupes $\varepsilon^*G$. Ainsi nous noterons $\gm$ le groupe multiplicatif et $\mu_n$ le groupe des racines $n$-i\`emes de l'unit\'e, que nous consid\`ererons indiff\'eremment comme des sch\'emas ou des log sch\'emas en groupes.

Soit $F$ un faisceau ab\'elien pour la topologie Kummer plate sur $X$. Par d\'erivation de foncteurs compos\'es, nous avons une suite spectrale
$$
H^p_{\fppf}(X,R^q\varepsilon_*F)\Longrightarrow H^{p+q}_{\kpl}(X,F).
$$

Consid\'erons le cas particulier o\`u $F=\varepsilon^*G$, avec $G$ un $X$-sch\'ema en groupes commutatif,
alors le morphisme d'adjonction $G\rightarrow \varepsilon_*\varepsilon^*G$ est un isomorphisme. On en d\'eduit, gr\^ace \`a la suite spectrale ci-dessus, une suite exacte
\begin{equation}
\label{spec1}
0\longrightarrow H^1_{\fppf}(X,G) \longrightarrow  H^1_{\kpl}(X,G) \longrightarrow  H^0(X,R^1\varepsilon_* G) \longrightarrow  H^2_{\fppf}(X,G)
\end{equation}

Dans les cas qui nous int\'eressent, le faisceau $R^1\varepsilon_* G$ a \'et\'e calcul\'e par Kato. Plus pr\'ecis\'ement, nous avons le r\'esultat suivant.

\begin{thm}
\label{epsilonkato}
Soient $X$ un log sch\'ema localement n\oe{}th\'erien, et $G$ un $X$-sch\'ema en groupes commutatif. On suppose que $G$ est fini localement libre sur $X$, ou que $G$ est affine lisse sur $X$. Alors on dispose d'un isomorphisme canonique
$$
R^1\varepsilon_* G \simeq \lim_{\longrightarrow} \Hom(\mu_n,G)\otimes_{\mathbb{Z}} \overline{M}_X^{\rm gp}
$$
la limite inductive \'etant prise par rapport aux projections canoniques $\mu_{mn}\rightarrow \mu_{n}$ avec $m,n\neq 0$.
\end{thm}

\begin{proof}
Voir \cite[Theorem 4.1]{kato2} ou \cite[Theorem 3.12]{niziol}.
\end{proof}

Rappelons que le foncteur (de la cat\'egorie $(fs/X)^{op}$ dans celle des groupes ab\'eliens)
$$
\gmlog:T\longmapsto \Gamma(T,M_T^{\rm gp})
$$
est un faisceau pour la topologie Kummer plate (voir \cite[Theorem 3.2]{kato2} ou \cite[Cor. 2.22]{niziol}). On dispose d'un monomorphisme canonique $\gm\rightarrow \gmlog$.

D'autre part, $\overline{M}_X^{\rm gp}$ d\'esigne l'image directe $\varepsilon_*(\gmlog/\gm)$, le quotient \'etant calcul\'e pour la topologie fppf. Les sections de $\overline{M}_X^{\rm gp}$ sur un $X$-sch\'ema $T\rightarrow X$ sont donn\'ees par $\Gamma(T,\overline{M}_T^{\rm gp})$, $T$ \'etant muni de la log structure image r\'eciproque de celle de $X$.

Le principal ingr\'edient de la preuve du th\'eor\`eme \ref{epsilonkato} est la suite exacte de Kummer en topologie log plate
$$
\begin{CD}
0 @>>> \mu_n @>>> \gmlog @>>> \gmlog @>>> 0 \\
\end{CD}
$$
induite par la multiplication par un entier $n\geq 1$ sur $\gmlog$. Celle-ci donne naissance \`a un cobord
\begin{equation}
\label{cobord1}
\begin{CD}
\delta:\gmlog(X)=M_X^{\rm gp}(X) @>>> H^1_{\kpl}(X,\mu_n) \\
\end{CD}
\end{equation}

Ainsi, \'etant donn\'es un morphisme $\mu_n\rightarrow G$ et une section globale de $M_X^{\rm gp}$, on peut leur associer un $G$-torseur via le morphisme naturel $H^1_{\kpl}(X,\mu_n)\rightarrow H^1_{\kpl}(X,G)$. En faisceautisant ce proc\'ed\'e, on obtient un morphisme \`a valeurs dans $R^1\varepsilon_* G$.

\begin{rmq}
Dans le cas o\`u $G=\gm$, on trouve
$$
R^1\varepsilon_* \gm = \mathbb{Q}/\mathbb{Z} \otimes \overline{M}_X^{\rm gp}
$$
Dans le cas o\`u $G=\mu_n$, il vient
$$
R^1\varepsilon_* \mu_n = \mathbb{Z}/n\mathbb{Z} \otimes \overline{M}_X^{\rm gp}
$$
\end{rmq}

\begin{exm}
\label{logpic}
Soit $X$ un sch\'ema n\oe{}th\'erien r\'egulier, muni de la log structure associ\'ee \`a un ouvert $U$ dont le compl\'ementaire est un diviseur $D=\sum_{m=1}^r D_m$ \`a croisements normaux \`a multiplicit\'es $1$ sur $X$ (cf. exemple \ref{debase}). Alors $\overline{M}_X^{\rm gp}(X)=\oplus_{m=1}^r \mathbb{Z}.D_m$ et la suite exacte \eqref{spec1} pour le groupe $\gm$ s'\'ecrit
$$
\begin{CD}
0 \longrightarrow \pic(X) @>>> H^1_{\kpl}(X,\gm) @>\nu >> \bigoplus_{m=1}^r (\mathbb{Q}/\mathbb{Z}).D_m \longrightarrow 0\\
\end{CD}
$$
la fl\`eche $\nu$ \'etant surjective en vertu de \cite[cor. 3.1.4]{gil5}. Nous verrons plus loin (cf. paragraphe \ref{torseurscycliques}) comment se traduit la suite exacte \eqref{spec1} pour le groupe $\mu_n$.
\end{exm}

On peut d\'eduire du th\'eor\`eme \ref{epsilonkato} le r\'esultat suivant de finitude.

\begin{cor}
Soient $X$ un log sch\'ema localement n\oe{}th\'erien, et $G$ un $X$-sch\'ema en groupes commutatif, fini localement libre. Alors le groupe $H^1_{\kpl}(X,G)$ est fini si et seulement si le groupe $H^1_{\fppf}(X,G)$ l'est.
\end{cor}

\begin{proof}
Comme $G$ est fini localement libre, la limite inductive des $\homr(\mu_n,G)$ est un groupe fini. D'autre part, $X$ \'etant un log sch\'ema fin et satur\'e, le faisceau $\overline{M}_X^{\rm gp}$ est un faisceau de groupes ab\'eliens libres de type fini. On en d\'eduit, via le th\'eor\`eme \ref{epsilonkato}, que $H^0(X,R^1\varepsilon_* G)$ est un groupe fini. Le r\'esultat d\'ecoule alors de la suite exacte \eqref{spec1}.
\end{proof}

Signalons ici un autre corollaire du th\'eor\`eme \ref{epsilonkato}, qui ach\`eve rapidement le cas des groupes \`a fibres unipotentes.

\begin{cor}
\label{groupeunipotent}
Soient $X$ un log sch\'ema localement n\oe{}th\'erien, et $G$ un $X$-sch\'ema en groupes commutatif fini et plat, \`a fibres unipotentes. Alors l'inclusion
$$
H^1_{\fppf}(X,G)\longrightarrow H^1_{\kpl}(X,G)
$$
est un isomorphisme.
\end{cor}

\begin{proof}
En effet, $G$ \'etant \`a fibres unipotentes, le faisceau $\Hom(\mu_n,G)$ est nul pour tout $n\geq 1$ en vertu de \cite[expos\'e XVII, lemme 2.5]{gro3t2}. On en d\'eduit, au vu du th\'eor\`eme \ref{epsilonkato}, que $R^1\varepsilon_* G=0$, d'o\`u le r\'esultat gr\^ace \`a la suite \eqref{spec1}.
\end{proof}

\begin{exm}
Si $X$ est un sch\'ema de caract\'eristique $p>0$, alors l'introduction des log structures ne cr\'ee pas de nouveaux torseurs sous les sch\'emas en groupes $\alpha_p$ et $\mathbb{Z}/p\mathbb{Z}$.
\end{exm}


\section{Torseurs sous un sch\'ema en groupes fini}
\label{trois}


Soient $X$ un log sch\'ema, et $G$ un $X$-sch\'ema en groupes commutatif, fini localement libre. En munissant $G$ de la log structure image r\'eciproque de celle de $X$, nous obtenons un log sch\'ema en groupes, que nous noterons encore $G$ comme nous l'avons signal\'e pr\'ec\'edemment. On v\'erifie ais\'ement (par adjonction) que le faisceau ab\'elien
$$
\Hom_X(G,\gm)
$$
sur le site $(fs/X)_{\kpl}$ est repr\'esentable par le dual de Cartier de $G$, que nous noterons $G^D$. Il est clair que nous avons un isomorphisme de bidualit\'e $G\simeq (G^D)^D$. Remarquons au passage que Kato a d\'efini dans \cite{kato3} des log sch\'emas en groupes qui sont, localement pour la topologie Kummer plate, repr\'esentables par des sch\'emas en groupes finis localement libres classiques. La dualit\'e de Cartier s'\'etend \`a ce cadre.

Le but de cette section est d'\'etudier les torseurs pour la topologie log plate sous des sch\'emas en groupes commutatifs finis classiques. On peut cependant remarquer que l'introduction des log structures a pour effet de rajouter de nouveaux faisceaux de groupes ab\'eliens. Il serait donc raisonnable de s'int\'eresser aux torseurs sous un log sch\'ema en groupes g\'en\'eral (\emph{i.e.} qui ne provient pas d'un sch\'ema en groupes classique).


\subsection{Questions de repr\'esentabilit\'e}

Le th\'eor\`eme de repr\'esentabilit\'e qui suit est d\^u \`a Kato (voir \cite[Theorem 9.1]{kato2}).

\begin{thm}
\label{representable}
Soit $X$ un log sch\'ema localement n\oe{}th\'erien. Soit $G$ un $X$-sch\'ema en groupes commutatif, fini localement libre, et soit $F$ un $G$-torseur pour la topologie log plate. Alors $F$ est repr\'esentable par un log sch\'ema qui est log plat et kumm\'erien sur $X$, et dont le sch\'ema sous-jacent est fini sur celui de $X$.
\end{thm}

Sous les hypoth\`eses du th\'eor\`eme \ref{representable}, le morphisme de sch\'emas sous-jacent \`a un torseur (pour la topologie log plate) est fini. Dans les cas qui nous int\'eressent, on peut m\^eme montrer qu'il est fini localement libre.

\begin{prop}
\label{loclibre}
Soit $X$ un log sch\'ema localement n\oe{}th\'erien. On suppose que $X$ admet, localement pour la topologie \'etale sur $X$, une carte model\'ee sur un mono\"ide de la forme $\mathbb{N}^r$. Soit $G$ un $X$-sch\'ema en groupes commutatif, fini localement libre, et soit $F$ un $G$-torseur (repr\'esentable) pour la topologie log plate. Alors le sch\'ema sous-jacent \`a $F$ est fini localement libre sur celui de $X$.
\end{prop}

\begin{proof}
Soit $f:Y\rightarrow X$ un $G$-torseur pour la topologie log plate. D'apr\`es le th\'eor\`eme \ref{representable}, le morphisme $f^{\circ}$ sous-jacent \`a $f$ est fini surjectif. La base \'etant localement n\oe{}th\'erienne, il suffit de montrer que $f^{\circ}$ est plat. Ceci est une question locale pour la topologie fppf sur $X$ et $Y$. D'apr\`es le lemme \ref{niziolstyle}, on peut donc supposer que $X$ admet une carte model\'ee sur $\mathbb{N}^r$, et que $f$ admet une carte $\mathbb{N}^r\rightarrow Q$ qui est un morphisme kumm\'erien de mono\"ides fs sans torsion. Le morphisme $f$ se d\'ecompose ainsi de la fa\c{c}on suivante
$$
Y\longrightarrow X\times_{\Spec(\mathbb{Z}[\mathbb{N}^r])} \Spec(\mathbb{Z}[Q]) \longrightarrow X
$$
o\`u le premier morphisme est strict et plat, et le second est fini localement libre --- donc plat --- sur les sch\'emas sous-jacents en vertu du lemme \ref{quisuit} ci-dessous. Par cons\'equent, $f^{\circ}$ est plat, ce qu'on voulait.
\end{proof}

\begin{lem}
\label{quisuit}
Soit $X$ un log sch\'ema muni d'une carte $\mathbb{N}^r\rightarrow M_X$. Alors, pour tout morphisme kumm\'erien $u:\mathbb{N}^r\rightarrow Q$ de mono\"ides fs sans torsion, le sch\'ema sous-jacent au torseur kumm\'erien standard associ\'e \`a $u$ est fini localement libre sur $X$.
\end{lem}

Avant de d\'emontrer ce lemme, nous commen\c{c}ons par rappeler la notion suivante. On dit qu'un morphisme $u:P\rightarrow Q$ entre mono\"ides fs est int\`egre si, pour tous $a_1$, $a_2$ dans $P$, $b_1$, $b_2$ dans $Q$ tels que $u(a_1)b_1=u(a_2)b_2$, il existe $a_3$, $a_4$ dans $P$ et $b$ dans $Q$ tels que $b_1=u(a_3)b$, $b_2=u(a_4)b$, et $a_1a_3=a_2a_4$.

\begin{proof}
Avec les notations du paragraphe \ref{pardeuxtrois}, le morphisme $X\rightarrow D(\mathbb{N}^r)$ est strict (car provenant d'une carte). Par cons\'equent, le morphisme de sch\'emas sous-jacent au torseur kumm\'erien standard associ\'e \`a $u$ n'est autre que le morphisme $\Spec(\mathbb{Z}[Q])\rightarrow \Spec(\mathbb{Z}[\mathbb{N}^r])$ induit par $u$, auquel on a inflig\'e le changement de base $X\rightarrow \Spec(\mathbb{Z}[\mathbb{N}^r])$ dans la cat\'egorie des sch\'emas. Ainsi, il suffit de montrer que le morphisme de sch\'emas $\Spec(\mathbb{Z}[Q])\rightarrow \Spec(\mathbb{Z}[\mathbb{N}^r])$ est fini localement libre. On sait d\'ej\`a que ce morphisme est fini ; la base $\Spec(\mathbb{Z}[\mathbb{N}^r])$ \'etant n\oe{}th\'erienne, reste \`a montrer qu'il est plat. D'apr\`es \cite[Proposition 4.1, (v)~$\Leftrightarrow$~(ii)]{kato1}, il suffit de montrer que le morphisme $u$ est injectif et int\`egre. Comme $u$ est kumm\'erien, il est injectif.

D'autre part, $Q$ \'etant sans torsion, $u^{\rm gp}:\mathbb{Z}^r\rightarrow Q^{\rm gp}$ est un morphisme injectif, \`a conoyau fini, entre deux $\mathbb{Z}$-modules libres. On peut alors trouver un entier $n\geq 1$ et un morphisme $v_0:Q^{\rm gp}\rightarrow \mathbb{Z}^r$ tel que $v_0\circ u^{\rm gp}$ soit la multiplication par $n$ dans $\mathbb{Z}^r$. En restreignant $v_0$ \`a $Q$, on obtient un morphisme $v:Q\rightarrow \mathbb{N}^r$ tel que $v\circ u$ soit la multiplication par $n$ dans le mono\"ide $\mathbb{N}^r$. On v\'erifie ais\'ement que $v$ est exact, et m\^eme kumm\'erien. Comme il est clair que $[n]:\mathbb{N}^r\rightarrow \mathbb{N}^r$ est un morphisme int\`egre, il suffit de montrer le fait g\'en\'eral suivant : si $u:P\rightarrow Q$ et $v:Q\rightarrow R$ sont deux morphismes de mono\"ides tels que $v\circ u$ est int\`egre et $v$ est exact, alors $u$ est int\`egre. Soient en effet $a_1,a_2\in P$ et $b_1,b_2\in Q$ tels que $u(a_1)b_1=u(a_2)b_2$. Alors $v(u(a_1))v(b_1)=v(u(a_2))v(b_2)$ et, puisque $v\circ u$ est int\`egre, il existe $a_3,a_4\in P$ et $r\in R$ tels que $v(b_1)=v(u(a_3))r$, $v(b_2)=v(u(a_4))r$, et $a_1a_3=a_2a_4$. Alors $v^{\rm gp}(u(a_3)^{-1}b_1)=r$ et, $v$ \'etant exact, on en d\'eduit que $u(a_3)^{-1}b_1$ appartient \`a $Q$. De m\^eme, $u(a_4)^{-1}b_2$ appartient \`a $Q$, et il vient :
$$
u(a_3)^{-1}b_1=u(a_4)^{-1}u(a_2)^{-1}u(a_1)b_1=u(a_4)^{-1}b_2
$$
On note $b$ cet \'el\'ement de $Q$, et l'on v\'erifie que ce $b$ satisfait les conditions voulues. Donc $u$ est int\`egre, ce qu'on voulait.
\end{proof}

\begin{rmq}
Le cas de la proposition \ref{loclibre} qui nous int\'eresse tout particuli\`erement est celui d'un sch\'ema r\'egulier, muni de la log structure associ\'ee \`a l'ouvert compl\'ementaire d'un diviseur \`a croisements normaux.
\end{rmq}

Comme autre application du th\'eor\`eme \ref{representable}, nous avons le r\'esultat suivant de comparaison de la cohomologie Kummer \'etale et de la cohomologie Kummer plate dans le cas o\`u $G$ est \'etale.

\begin{cor}
Soit $X$ un log sch\'ema localement n\oe{}th\'erien. Soit $G$ un $X$-sch\'ema en groupes commutatif, fini \'etale. Alors le morphisme canonique
$$
H^1_{\ket}(X,G)\longrightarrow H^1_{\kpl}(X,G)
$$
est un isomorphisme.
\end{cor}

\begin{proof}
Par la suite spectrale comparant les topologies log plate et log \'etale, on sait que ce morphisme est injectif, montrons sa surjectivit\'e. Soit $F$ un $G$-torseur pour la topologie Kummer log plate, alors $F$ est repr\'esentable par un log sch\'ema $X'\rightarrow X$ qui est log plat, kumm\'erien et surjectif d'apr\`es le th\'eor\`eme \ref{representable}. Par descente log plate (voir \cite[Theorem 7.1]{kato2}), on en d\'eduit que $X'\rightarrow X$ est log \'etale.
\end{proof}


\subsection{Restriction et prolongement}

Le but de cette section est de montrer, sous des hypoth\`eses raisonnables, l'unicit\'e du prolongement d'un torseur (pour la topologie log plate) sous un sch\'ema en groupes fini localement libre.

\begin{prop}
\label{injectrestr}
Soit $X$ un log sch\'ema localement n\oe{}th\'erien int\`egre normal, et soit $G$ un $X$-sch\'ema en groupes commutatif, fini localement libre. Soit $U$ un ouvert dense du sch\'ema $X$, que l'on munit de la log structure induite par celle de $X$. Alors le morphisme de restriction
$$
H^1_{\kpl}(X,G)\longrightarrow H^1_{\kpl}(U,G)
$$
est injectif.
\end{prop}

\begin{proof}
Tout d'abord, observons le fait suivant : si $g:T\rightarrow X$ est un $G$-torseur (repr\'esentable) pour la topologie fppf, alors $g^{\sharp}:T^{\sharp}\rightarrow X$ muni de la log structure image r\'eciproque de celle de $X$ est l'unique $G$-torseur pour la topologie log plate ayant $g$ pour morphisme de sch\'emas sous-jacent. En effet, soit $T'\rightarrow X$ un $G$-torseur log plat ayant $g$ pour morphisme sous-jacent, alors par adjonction l'identit\'e de $T$ induit un morphisme de $X$-log sch\'emas $T'\rightarrow T^{\sharp}$ qui commute \`a l'action du groupe $G$, c'est-\`a-dire qui est un morphisme de torseurs. Par cons\'equent, c'est un isomorphisme.

Soit $f:Y\rightarrow X$ un $G$-torseur pour la topologie Kummer plate (forc\'ement repr\'esentable d'apr\`es le th\'eor\`eme \ref{representable}). 
Pour montrer que ce torseur est trivial, il suffit, d'apr\`es le raisonnement qui pr\'ec\`ede, de montrer que le morphisme $f^{\circ}$ sous-jacent \`a $f$ est un $G$-torseur trivial pour la topologie fppf, ce qui revient \`a exhiber une section de $f^{\circ}$.

Supposons que $f_U:Y_U\rightarrow U$ soit trivial, et montrons que $f$ est trivial. D'apr\`es le th\'eor\`eme \ref{representable}, le morphisme sous-jacent $f^{\circ}$ est fini, donc entier. Par suite, $X$ \'etant int\`egre normal, toute section $U\rightarrow Y_U$ de $f^{\circ}$ au-dessus de $U$ se prolonge en une section de $f^{\circ}$ au-dessus de $X$, en vertu de \cite[6.1.14]{ega2}. D'o\`u le r\'esultat.
\end{proof}

Supposons, avec les hypoth\`eses et notations de la proposition \ref{injectrestr}, que $U$ soit l'ouvert de trivialit\'e de $X$. Alors le morphisme de restriction s'\'ecrit
$$
H^1_{\kpl}(X,G)\longrightarrow H^1_{\fppf}(U,G)
$$
et on se demande \`a pr\'esent (cf. paragraphe \ref{hunhun}) quelle est l'image de ce morphisme.

La philosophie g\'en\'erale adopt\'ee ici s'inspire du r\'esultat suivant (reformulation de \cite[th\'eor\`eme 7.6]{illusie}), lequel sugg\`ere que l'image devrait \^etre constitu\'ee des torseurs qui sont mod\'er\'ement ramifi\'es  au-dessus du compl\'ementaire de $U$, pour autant qu'on puisse donner un sens pr\'ecis \`a cette phrase.

\begin{thm}
\label{GMtame}
Soit $X$ un log sch\'ema log r\'egulier, et soit $G$ un $X$-sch\'ema en groupes commutatif, fini \'etale. Alors l'ouvert de trivialit\'e $U$ est dense dans $X$, et le morphisme de restriction
$$
H^1_{\ket}(X,G) \longrightarrow H^1_{\et}(U,G_U)
$$
est injectif. Son image est l'ensemble des $G_U$-torseurs $V\rightarrow U$ qui sont mod\'er\'ement ramifi\'es le long de $X\backslash U$, c'est-\`a-dire tels que le normalis\'e de $X$ dans $V$ soit mod\'er\'ement ramifi\'e au-dessus des points de $X\backslash U$ de codimension $1$.
\end{thm}

La notion de ramification utilis\'ee dans ce th\'eor\`eme est la notion usuelle pour un morphisme fini de traits (voir \cite[2.2.2 et 2.2.3]{gm}).

\begin{rmq}
\label{GMlike}
$(a)$ Si $X$ est un sch\'ema r\'egulier, muni de la log structure associ\'ee \`a un ouvert dont le compl\'ementaire est un diviseur \`a croisements normaux $D$, alors il est log r\'egulier en tant que log sch\'ema, et le th\'eor\`eme ci-dessus s'applique.

$(b)$ Sous les hypoth\`eses de $(a)$, les $G$-torseurs sur $X$ pour la topologie Kummer log \'etale sont exactement les rev\^etements de $X$ mod\'er\'ement ramifi\'es relativement \`a $D$ (au sens de \cite[Def. 2.2.2]{gm}) dont la restriction \`a $U$ est un $G$-torseur pour la topologie \'etale. En particulier, si $Y\rightarrow X$ est un $G$-torseur pour la topologie log \'etale, alors le sch\'ema sous-jacent \`a $Y$ est normal (donc est le normalis\'e de $X$ dans $Y_U$).
\end{rmq}

Dans la m\^eme veine, la proposition ci-dessous (voir \cite[prop. 3.2.1]{gil5}) d\'ecrit une situation dans laquelle le morphisme de restriction est un isomorphisme. Ceci confirme l'intuition qu'un groupe de type multiplicatif agit toujours de fa\c{c}on mod\'er\'ee.

\begin{prop}
\label{dejafait}
Soit $X$ un sch\'ema n\oe{}th\'erien r\'egulier, muni de la log structure associ\'ee \`a un ouvert $U$ dont le compl\'ementaire est un diviseur \`a croisements normaux sur $X$. Alors, pour tout entier $n\geq 1$, le morphisme de restriction
$$
H^1_{\kpl}(X,\mu_n)\longrightarrow H^1_{\fppf}(U,\mu_n)
$$
est un isomorphisme.
\end{prop}


\subsection{Rev\^etements cycliques uniformes}
\label{torseurscycliques}

Nous rappelons bri\`evement en quoi consistent les rev\^etements cycliques uniformes. Nous nous appuyons pour cela sur les travaux de Arsie et Vistoli \cite{av}.

Pour toute la dur\'ee de ce paragraphe, nous fixons un entier naturel $n\geq 1$.

\begin{dfn}
\label{cyclicunif}
Soit $X$ un sch\'ema n\oe{}th\'erien. Un rev\^etement cyclique uniforme de degr\'e $n$ de $X$ est un morphisme de sch\'emas $Y\rightarrow X$ avec une action de $\mu_n$ sur $Y$, satisfaisant la condition suivante : tout point de $X$ admet un voisinage ouvert affine $V=\Spec(R)\subseteq X$ tel que le morphisme $Y_V\rightarrow V$ soit de la forme $\Spec(R[t]/(t^n-a))\rightarrow\Spec(R)$ avec $a$ un \'el\'ement de $R$ qui n'est pas un diviseur de $0$, et tel que l'action de $\mu_n$ sur $Y$ induise l'action \'evidente sur $\Spec(R[t]/(t^n-a))$.
\end{dfn}

Si $f:Y\rightarrow X$ est un rev\^etement cyclique uniforme, on peut d\'efinir son diviseur de branchement $\Delta_f\subseteq X$ comme \'etant l'ensemble des couples $(V,a)$ comme dans la d\'efinition, pour $V$ variant sur un recouvrement de $X$ convenablement choisi. De m\^eme, le diviseur de ramification $D_f\subseteq Y$ est l'ensemble des couples $(Y_V,t)$. On constate aussit\^ot la relation $f^*\Delta_f=nD_f$ entre les deux diviseurs.

Le th\'eor\`eme ci-dessous donne une correspondance entre $\mu_n$-torseurs pour la topologie log plate et rev\^etements cycliques uniformes.

\begin{thm}
\label{logcyclic}
Soit $X$ un sch\'ema n\oe{}th\'erien r\'egulier, muni de la log structure associ\'ee \`a un ouvert $U$ dont le compl\'ementaire $D$ est un diviseur \`a croisements normaux sur $X$. Soit $Y\rightarrow X$ un rev\^etement cyclique uniforme de degr\'e $n$, dont le support du diviseur de branchement est contenu dans le support de $D$. On munit $Y$ de la log structure associ\'ee \`a l'ouvert compl\'ementaire du diviseur de ramification. Alors $Y\rightarrow X$ est un $\mu_n$-torseur pour la topologie log plate. R\'eciproquement, tout $\mu_n$-torseur sur $X$ pour la topologie log plate est obtenu de la sorte.
\end{thm}

\begin{proof}
Soit $Y\rightarrow X$ un rev\^etement cyclique uniforme de degr\'e $n$, dont le support du diviseur de branchement est contenu dans le compl\'ementaire de $U$. Quitte \`a travailler localement pour la topologie de Zariski sur $X$, on peut supposer que le morphisme $Y\rightarrow X$ est de la forme $\Spec(R[t]/(t^n-a))\rightarrow\Spec(R)$ avec $a$ un \'el\'ement de $R$ qui n'est pas un diviseur de $0$. Alors le diviseur de branchement est d\'efini par $a$, lequel est par cons\'equent un \'el\'ement inversible dans $\Gamma(U,\mathcal{O}_U)$. On en d\'eduit que $\Spec(R[t]/(t^n-a))\rightarrow\Spec(R)$ est un torseur pour la topologie Kummer log plate, le sch\'ema du haut \'etant muni de la log structure associ\'ee \`a l'ouvert compl\'ementaire du diviseur d\'efini par $t$ (c'est-\`a-dire au diviseur de ramification). Plus pr\'ecis\'ement, ce torseur est l'image de $a$ par le cobord $\delta:\Gamma(U,\mathcal{O}_U^*)=M_X^{\rm gp}(X)\longrightarrow H^1_{\kpl}(X,\mu_n)$ de la suite de Kummer pour $\gmlog$ (voir \eqref{cobord1}).

R\'eciproquement, soit $Y\rightarrow X$ un $\mu_n$-torseur pour la topologie log plate. Quitte \`a travailler localement pour la topologie de Zariski sur $X$, on peut supposer que $D$ est un diviseur dont les composantes irr\'eductibles $D_1,\dots, D_r$ sont principales, disons engendr\'ees par $b_1,\dots,b_r\in \Gamma(X,\mathcal{O}_X)$. Alors la suite exacte \eqref{spec1} pour le groupe $G=\mu_n$ s'\'ecrit
$$
\begin{CD}
0\longrightarrow H^1_{\fppf}(X,\mu_n) @>>>  H^1_{\kpl}(X,\mu_n) @>\nu_n>> \bigoplus_{m=1}^r (\mathbb{Z}/n\mathbb{Z}).D_m
\end{CD}
$$
le groupe $\overline{M}_X^{\rm gp}(X)$ \'etant canoniquement isomorphe \`a $\bigoplus_{m=1}^r \mathbb{Z}.D_m$ (voir l'exemple \ref{logpic}). Soit $(k_1,\dots, k_r)$ un repr\'esentant entier de $\nu_n(Y)$, alors $\delta(b_1^{k_1}\dots b_r^{k_r})$ et $Y$ ont la m\^eme image par $\nu_n$. C'est-\`a-dire que le produit contract\'e $Y\vee_{\mu_n} \delta(b_1^{k_1}\dots b_r^{k_r})^{-1}$ est un $\mu_n$-torseur pour la topologie fppf. Quitte \`a se localiser encore une fois pour la topologie de Zariski sur $X$, on peut supposer que ce dernier est de la forme $\delta(\alpha)$ avec $\alpha\in \Gamma(X,\mathcal{O}_X^*)$. On constate alors que $Y=\delta(\alpha b_1^{k_1}\dots b_r^{k_r})$ et par cons\'equent le sch\'ema sous-jacent \`a $Y$ est le spectre de $\mathcal{O}_X[t]/(t^n-\alpha b_1^{k_1}\dots b_r^{k_r})$. Ainsi $Y\rightarrow X$ est un rev\^etement cyclique uniforme.
\end{proof}

Dans \cite[section 2]{av} on donne une description alternative des rev\^etements cycliques uniformes, que nous allons d\'etailler.

Soit $\RAC(X,n)$ la cat\'egorie suivante. Les objets de $\RAC(X,n)$ sont les couples $(\mathcal{L},\phi)$ o\`u $\mathcal{L}$ est un faisceau inversible sur $X$ et $\phi:\mathcal{L}^{\otimes n}\rightarrow \mathcal{O}_X$ est un monomorphisme.
Un morphisme entre deux couples $(\mathcal{L},\phi)$ et $(\mathcal{N},\psi)$ est la donn\'ee d'un isomorphisme $\kappa:\mathcal{L}\rightarrow \mathcal{N}$ tel que $\psi\circ \kappa^{\otimes n}=\phi$. On obtient ainsi un groupo\"ide (tous les morphismes sont des isomorphismes). En outre, on peut d\'efinir le produit de deux objets  $(\mathcal{L},\phi)$ et $(\mathcal{N},\psi)$ comme \'etant le couple $(\mathcal{L}\otimes \mathcal{N},m\circ(\phi\otimes\psi))$, o\`u $m:\mathcal{O}_X\otimes\mathcal{O}_X\rightarrow\mathcal{O}_X$ est l'isomorphisme canonique induit par la multiplication de $\mathcal{O}_X$. L'\'el\'ement neutre est le couple $(\mathcal{O}_X,i_X)$ o\`u $i_X:\mathcal{O}_X^{\otimes n}\simeq \mathcal{O}_X$ est l'isomorphisme canonique.

La cat\'egorie $\RAC(X,n)$,  munie de ce produit, est une cat\'egorie de Picard strictement commutative (au sens de \cite[expos\'e XVIII, 1.4.2]{gro4t3}). On note $\Rac(X,n)$ le groupe des classes d'isomorphie d'objets de cette cat\'egorie.

\'Etant donn\'e un objet $(\mathcal{L},\phi)$ de $\RAC(X,n)$, on peut lui associer une $\mathcal{O}_X$-alg\`ebre
$$
\Sa(\mathcal{L},\phi):=\Sym(\mathcal{L})/(\phi:\mathcal{L}^{\otimes n}\rightarrow \mathcal{O}_X)=\mathcal{O}_X\oplus \mathcal{L}\oplus\dots\oplus\mathcal{L}^{\otimes (n-1)}
$$
(o\`u $\Sym(\mathcal{L})$ d\'esigne l'alg\`ebre sym\'etrique de $\mathcal{L}$) dans laquelle la loi de multiplication est induite par le produit tensoriel et par l'application $\phi$. Le spectre de $\Sa(\mathcal{L},\phi)$ est un rev\^etement cyclique uniforme de degr\'e $n$ de $X$.
L'image de $\phi$ est l'id\'eal d\'efinissant le diviseur de branchement de ce rev\^etement.

On v\'erifie ais\'ement (voir \cite[section 2]{av}) que tout rev\^etement cyclique uniforme de degr\'e $n$ de $X$ est obtenu de la sorte. Ainsi, le diviseur de branchement d'un rev\^etement cyclique uniforme de degr\'e $n$ est canoniquement muni d'une racine $n$-i\`eme.

Soient $D$ un diviseur sur $X$, et $U$ son ouvert compl\'ementaire. On note $\Rac(X,D,n)$ le sous-groupe de $\Rac(X,n)$ constitu\'e des couples $(\mathcal{L},\phi)$ tels que le support du diviseur d\'efini par le faisceau d'id\'eaux image de $\phi$ soit contenu dans le support de $D$.

Le groupe $\Rac(X,\emptyset,n)$ est l'ensemble des couples $(\mathcal{L},\phi)$ tels que $\phi$ soit un isomorphisme. Il est bien connu qu'un tel couple d\'efinit un $\mu_n$-torseur pour la topologie fppf, en d'autres termes $\Rac(X,\emptyset,n)$ est isomorphe \`a $H^1_{\fppf}(X,\mu_n)$.
Le corollaire ci-dessous constitue une g\'en\'eralisation de ce r\'esultat en topologie Kummer log plate.

\begin{cor}
\label{logcyclor}
Soit $X$ un sch\'ema n\oe{}th\'erien r\'egulier, muni de la log structure associ\'ee \`a un ouvert $U$ dont le compl\'ementaire $D$ est un diviseur \`a croisements normaux sur $X$. Alors l'application
$$
\mathfrak{S}:\Rac(X,D,n)\longrightarrow H^1_{\kpl}(X,\mu_n)
$$
qui envoie un couple $(\mathcal{L},\phi)$ sur le spectre de l'alg\`ebre $\Sa(\mathcal{L},\phi)$, est un morphisme surjectif de groupes.
\end{cor}

\begin{rmq}
Le morphisme $\mathfrak{S}$  n'est pas en g\'en\'eral un isomorphisme, bien qu'il le soit dans le cas o\`u $D=\emptyset$. Pour \'eclaircir la situation, consid\'erons le diagramme commutatif suivant
$$
\begin{CD}
\Rac(X,D,n) @>\mathfrak{S}>> H^1_{\kpl}(X,\mu_n) \\
@Vr_U VV @VVV \\
\Rac(U,\emptyset,n) @>\mathfrak{S}_U>> H^1_{\fppf}(U,\mu_n) \\
\end{CD}
$$
dans lequel les fl\`eches verticales sont les morphismes de restriction. On sait que $\mathfrak{S}_U$ est un isomorphisme, et que la fl\`eche verticale de droite est injective (c'est m\^eme un isomorphisme d'apr\`es la proposition \ref{dejafait}). On en d\'eduit que le noyau de $\mathfrak{S}$ est \'egal au noyau de $r_U$.

Nous allons construire un exemple. Supposons que $I\subseteq \mathcal{O}_X$ soit un id\'eal localement libre de rang $1$, d\'efinissant un diviseur dont le support est contenu dans celui de $D$. Soit $\iota:I^{\otimes n}\simeq I^n\subseteq \mathcal{O}_X$ le monomorphisme canonique. Alors le couple $(I,\iota)$ est un \'el\'ement de $\Rac(X,D,n)$. De plus, l'image du couple $(I,\iota)$ par l'application $r_U$ est le couple $(\mathcal{O}_U,i_U)$ o\`u $i_U:\mathcal{O}_U^{\otimes n}\simeq \mathcal{O}_U$ est l'isomorphisme canonique. Autrement dit, $r_U(I,\iota)$ est l'\'el\'ement neutre de $\Rac(U,\emptyset,n)$, et $\mathfrak{S}(I,\iota)$ est le torseur trivial.

Enfin, il n'est pas difficile de voir que la classe d'isomorphie du couple $(I,\iota)$ est l'\'el\'ement neutre de $\Rac(X,D,n)$ si et seulement si $I$ est un id\'eal principal. Il suffit donc de choisir un id\'eal $I$ non principal pour obtenir un \'el\'ement non nul dans le noyau de $\mathfrak{S}$. Un tel choix est possible d\`es lors que $D$ n'est pas un diviseur principal.
\end{rmq}

Sous les hypoth\`eses de \ref{logcyclic} et \ref{logcyclor} (et avec les notations de l'exemple \ref{logpic}), le compl\'ementaire de $U$ est un diviseur $D=\sum_{m=1}^r D_m$ \`a croisements normaux \`a multiplicit\'es $1$ sur $X$. Nous avons d\'ecrit dans \cite[prop. 3.2.2]{gil5} une suite exacte
$$
\begin{CD}
0 \rightarrow H^1_{\fppf}(X,\mu_n)\rightarrow H^1_{\kpl}(X,\mu_n) @>\nu_n>> \bigoplus_{m=1}^r (\frac{1}{n}\mathbb{Z}/\mathbb{Z}).D_m
@>\theta_n>> \pic(X)/n \rightarrow \pic(U)/n \\
\end{CD}
$$
qui est un avatar de la suite exacte \eqref{spec1} pour le groupe $\mu_n$. Nous allons expliciter ici le morphisme $\nu_n$. Soit $T\rightarrow X$ un $\mu_n$-torseur log plat. En vertu du th\'eor\`eme \ref{logcyclic}, le morphisme de sch\'emas sous-jacent est un rev\^etement cyclique uniforme, dont le support du diviseur de branchement $\Delta_{T/X}$ est contenu dans celui de $D$. On peut donc \'ecrire
$$
\Delta_{T/X}=\sum_{m=1}^r k_m.D_m
$$
o\`u les $k_m$ sont des entiers, bien d\'efinis modulo $n$ d'apr\`es ce qui pr\'ec\`ede. Alors $\nu_n(T)$ est tout simplement $\Delta_{T/X}$ vu en tant que diviseur \`a coefficients dans $\frac{1}{n}\mathbb{Z}/\mathbb{Z}$ (c'est-\`a-dire que l'on remplace $k_m$ par la classe de $k_m/n$ modulo $\mathbb{Z}$).
D'autre part, sachant (cf. corollaire \ref{logcyclor}) que $T$ provient d'un couple $(\mathcal{L},\phi)\in\Rac(X,D,n)$, on obtient l'interpr\'etation suivante : $T$ r\'esulte de l'extraction d'une racine $n$-i\`eme de son diviseur de branchement $\Delta_{T/X}=\nu_n(T)$. Ceci clarifie la situation \'evoqu\'ee dans \cite[Remarque 3.2.3]{gil5}.

\begin{rmq}
\'Etant donn\'es un sch\'ema r\'egulier $X$ et un rev\^etement cyclique uniforme $f:Y\rightarrow X$ de degr\'e $n$, peut-on prolonger $f$ dans la cat\'egorie des log sch\'emas de telle sorte que le morphisme obtenu soit un $\mu_n$-torseur pour la topologie log plate ? Si le diviseur de branchement est \`a croisements normaux, alors le th\'eor\`eme \ref{logcyclic} donne une r\'eponse positive \`a cette question. En revanche, si tel n'est pas le cas, alors la log structure correspondante sur $X$ n'est pas toujours fine. Il faudrait travailler avec des log structures pour la topologie de Zariski (au lieu de la topologie \'etale) pour esp\'erer obtenir un torseur log plat. Nous avons pr\'ef\'er\'e nous cantonner ici au cas de log structures \'etales, lequel cas \'etant le plus courant dans la litt\'erature.
\end{rmq}


\subsection{Lien avec les actions mod\'er\'ees}

Le but de cette section est d'expliciter le lien entre la notion de torseur pour la topologie Kummer log plate et la notion d'action mod\'er\'ee d'un sch\'ema en groupes sur un sch\'ema introduite par Chinburg, Erez, Pappas et Taylor dans \cite{cept}.

\begin{dfn}
Soient $S$ un sch\'ema affine et $G$ un $S$-sch\'ema en groupes commutatif, fini localement libre. On se donne un $S$-sch\'ema $Y$, muni d'une action de $G$.
\begin{enumerate}
\item[$(1)$] Nous dirons que l'action de $G$ sur $Y$ est mod\'er\'ee si elle l'est au sens de \cite[Definition 7.1]{cept}.
\item[$(2)$] Supposons que $Y$ soit affine, et que l'alg\`ebre de $Y$ soit localement libre de m\^eme rang que celle de $G$ en tant que module sur l'alg\`ebre de $S$. Nous dirons que l'action de $G$ sur $Y$ est CH-mod\'er\'ee si elle l'est au sens de \cite[2.\textit{f}]{cept}.
\end{enumerate}
\end{dfn}

La notion d'action CH-mod\'er\'ee a \'et\'e introduite par Childs et Hurley dans \cite{ch}. Leurs travaux sont formul\'es dans le langage des alg\`ebres de Hopf.

\begin{thm}
\label{CEPTtame}
Soient $S$ un sch\'ema affine et $G$ un $S$-sch\'ema en groupes commutatif, fini localement libre. Soit $X$ un $S$-log sch\'ema localement n\oe{}th\'erien, et soit $Y\rightarrow X$ un $G$-torseur pour la topologie log plate. Alors l'action de $G$ sur le sch\'ema sous-jacent \`a $Y$ est mod\'er\'ee.
\end{thm}

\begin{proof}
On peut supposer que $X$ est affine. Comme $Y\rightarrow X$ est un torseur, $X$ est le quotient de $Y$ par l'action de $G$. Par d\'efinition, ce quotient est le conoyau, dans la cat\'egorie des log sch\'emas, des fl\`eches $m, pr_1:Y\times_S G\rightarrow Y$, o\`u $m$ est l'action de $G$ et $pr_1$ la premi\`ere projection. D'autre part, le foncteur d'oubli des log structures admet un adjoint \`a droite, donc pr\'eserve les conoyaux. Soit $Y^{\circ}$ le sch\'ema sous-jacent \`a $Y$. Comme $G$ est muni de la log structure triviale, le sch\'ema sous-jacent \`a $Y\times_S G$ est $Y^{\circ}\times_S G$. On en d\'eduit que le sch\'ema sous-jacent \`a $X$ est bien le quotient, dans la cat\'egorie des sch\'emas, du sch\'ema sous-jacent \`a $Y$ par l'action de $G$.

D'autre part, le th\'eor\`eme \ref{representable} nous dit que $Y\rightarrow X$ est un morphisme fini. Dans ce cadre, le fait que l'action de $G$ sur $Y$ soit mod\'er\'ee est une notion locale pour la topologie fppf sur $X$ d'apr\`es \cite[Proposition 2.24]{cept}.

Soit $n\geq 1$ un entier naturel. Nous avons une application
$$
b_n:\homr(\mu_n,G)\times M_X^{\rm gp}(X)\longrightarrow H^1_{\kpl}(X,G)
$$
introduite par Kato dans sa d\'emonstration de \cite[Theorem 4.1]{kato2} (voir aussi la preuve de Nizio\l{} \cite[Theorem 3.12]{niziol}), et que nous allons d\'ecrire. Soit $h:\mu_n\rightarrow G$ un morphisme de $X$-sch\'emas en groupes. Nous avons deux morphismes
$$
\begin{CD}
M_X^{\rm gp}(X) @>\delta>> H^1_{\kpl}(X,\mu_n) @>H^1(h)>> H^1_{\kpl}(X,G) \\
\end{CD}
$$
o\`u $\delta$ est le cobord \eqref{cobord1} de la suite de Kummer pour $\gmlog$ et $H^1(h)$ est le morphisme induit par $h$. Soit d'autre part $a$ une section globale de $M_X^{\rm gp}$. Alors on pose, par d\'efinition
$$
b_n(h,a):=H^1(h)(\delta(a)).
$$
D'apr\`es le th\'eor\`eme \ref{epsilonkato} (ou plut\^ot, d'apr\`es sa preuve), localement pour la topologie fppf sur $X$, il existe un entier $n$ tel que le torseur $Y\rightarrow X$ soit dans l'image de l'application $b_n$. On se ram\`ene ainsi au cas o\`u l'action de $G$ sur $Y$ provient d'une action de $\mu_n$. Or toute action induite par une action d'un groupe fini diagonalisable est mod\'er\'ee d'apr\`es \cite[Lemma 2.5]{cept} (notons que les auteurs se limitent au cas o\`u $h:\mu_n\rightarrow G$ est une immersion ferm\'ee ; on peut toujours s'y ramener en quotientant $\mu_n$ par le noyau de $h$, le r\'esultat \'etant un groupe de la forme $\mu_q$ pour un certain $q$ divisant $n$. Le monomorphisme obtenu $\overline{h}:\mu_q\rightarrow G$ est une immersion ferm\'ee en vertu de \cite[expos\'e VIII, cor. 5.7]{gro3t2}). On en d\'eduit le r\'esultat.
\end{proof}

\begin{rmq}
Le th\'eor\`eme \ref{CEPTtame} permet de traduire certains r\'esultats de \cite{cept} en termes de log sch\'emas. En particulier notre corollaire \ref{groupeunipotent} est \`a mettre en parall\`ele avec la Proposition 6.2 de \cite{cept}, selon laquelle toute action mod\'er\'ee d'un groupe \`a fibres unipotentes d\'efinit un torseur fppf.
\end{rmq}

\begin{cor}
\label{CHtame}
Soit $X$ un log sch\'ema localement n\oe{}th\'erien. On suppose que $X$ admet, localement pour la topologie \'etale sur $X$, une carte model\'ee sur un mono\"ide de la forme $\mathbb{N}^r$, et que l'ouvert de trivialit\'e est dense dans $X$. Soit $G$ un $X$-sch\'ema en groupes commutatif fini localement libre, et soit $Y\rightarrow X$ un $G$-torseur pour la topologie log plate. Alors
\begin{enumerate}
\item[$(1)$] Si $X$ est affine, l'action de $G$ sur le sch\'ema sous-jacent \`a $Y$ est CH-mod\'er\'ee.
\item[$(2)$] L'alg\`ebre sous-jacente \`a $Y$ est localement libre de rang un en tant que module sur l'alg\`ebre de $G^D$.
\end{enumerate}
\end{cor}

\begin{proof}
$(1)$ D'apr\`es le th\'eor\`eme \ref{CEPTtame}, l'action de $G$ sur le sch\'ema sous-jacent \`a $Y$ est mod\'er\'ee, de quotient $Y\rightarrow X$. D'apr\`es la proposition \ref{loclibre}, le morphisme $Y\rightarrow X$ est fini localement libre, de m\^eme rang que $G$ (en effet, sa restriction \`a l'ouvert dense de trivialit\'e de $X$ est un $G$-torseur fppf, donc est de m\^eme rang que $G$). Au vu de tout cela et par application de \cite[Theorem 2.6]{cept}, l'action de $G$ sur $Y$ est CH-mod\'er\'ee. $(2)$ La question \'etant locale pour la topologie de Zariski sur $X$, on peut supposer que $X$ est affine. L'action sous-jacente au torseur est alors CH-mod\'er\'ee d'apr\`es ce qui pr\'ec\`ede. Par cons\'equent, en vertu de \cite[Proposition 2.7]{cept}, l'alg\`ebre sous-jacente \`a $Y$ est  projective, et m\^eme localement libre de rang un, sur l'alg\`ebre de $G^D$.
\end{proof}



\section{Structure galoisienne des torseurs}
\label{quatre}


Dor\'enavant, nous fixons un log sch\'ema fin et satur\'e $X$, dont le sch\'ema sous-jacent est localement n\oe{}th\'erien.

\subsection{Nullit\'e d'un faisceau Ext local}

Tout d'abord, nous \'enon\c{c}ons le th\'eor\`eme suivant, qui g\'en\'eralise un r\'esultat bien connu en topologie fppf.

\begin{thm}
\label{splitext}
Soit $G$ un $X$-sch\'ema en groupes commutatif, fini localement libre. Alors le faisceau $\Ext^1_{\text{\rm kpl}}(G,\gm)$ est nul.
\end{thm}

\begin{proof}
Il s'agit de montrer que toute extension (dans la cat\'egorie des faisceaux ab\'eliens sur $(fs/X)_{\kpl}$) de $G$ par $\gm$ admet une section localement pour la topologie Kummer plate sur $X$. Soit $\Omega$ une telle extension. Soit $n$ l'ordre de $G$, on sait que $G$ est tu\'e par $n$. En se servant de la suite exacte de Kummer donn\'ee par la multiplication par $n$ sur $\gm$, on en d\'eduit que le morphisme naturel
$$
\ext^1_{\kpl}(G,\mu_n)\longrightarrow\ext^1_{\kpl}(G,\gm)
$$
est surjectif. On fixe \`a pr\'esent une extension de $G$ par $\mu_n$ qui est un ant\'ec\'edent de $\Omega$ par le morphisme ci-dessus. Cette extension est donn\'ee par une suite exacte
$$
\begin{CD}
0 @>>> \mu_n @>>> F @>>> G @>>> 0\\
\end{CD}
$$
Par un argument cassique (valable dans n'importe quel topos), $F$ est un $\mu_n$-torseur sur $G$. D'apr\`es le th\'eor\`eme \ref{representable}, $X$ et $G$ \'etant localement n\oe{}th\'eriens, $F$ est donc repr\'esentable par un log sch\'ema fs, et le morphisme $F\rightarrow G$ est log plat et kumm\'erien.

Quitte \`a travailler localement pour la topologie \'etale sur $X$, on peut supposer que $X$ est affine, et admet une carte globale $P\rightarrow M_X$. Comme la log structure de $G$ est l'image r\'eciproque de celle de $X$, on obtient ainsi une carte globale $P\rightarrow M_G$ pour $G$, model\'ee sur le m\^eme mono\"ide $P$.

D'apr\`es Kato (voir \cite[Proposition 2.7]{kato2}), le sch\'ema $G$ \'etant quasi-compact (car affine), il existe un mono\"ide fs $Q$ et un morphisme kumm\'erien $u:P\rightarrow Q$ tel que, si
$$
T:=G\times_{\Spec(\mathbb{Z}[P])} \Spec(\mathbb{Z}[Q]) \longrightarrow G
$$
est le torseur kumm\'erien standard (de base $G$) associ\'e \`a $u$, le morphisme
$$
F\times_G T\longrightarrow T
$$
d\'eduit du morphisme $F\rightarrow G$ par changement de base $T\rightarrow G$, est strict. Celui-ci se r\'ecrit aussit\^ot
$$
F\times_{\Spec(\mathbb{Z}[P])} \Spec(\mathbb{Z}[Q])\longrightarrow G\times_{\Spec(\mathbb{Z}[P])} \Spec(\mathbb{Z}[Q])
$$
c'est-\`a-dire
$$
F\times_X Y\longrightarrow G\times_X Y
$$
o\`u $Y:=X\times_{\Spec(\mathbb{Z}[P])} \Spec(\mathbb{Z}[Q])\rightarrow X$ est le torseur kumm\'erien standard (de base $X$) associ\'e \`a $u$. On obtient ainsi, en infligeant le changement de base $Y\rightarrow X$ \`a la suite exacte d\'efinissant $F$, une suite exacte de $Y$-log sch\'emas dont les log structures sont obtenues par image r\'eciproque de la log structure de $Y$. Autrement dit, apr\`es changement de base par un \'epimorphisme Kummer log plat, notre extension $F$ de $G$ par $\mu_n$ pour la topologie log plate provient d'une extension de $G$ par $\mu_n$ pour la topologie fppf. Mais alors, apr\`es application du m\^eme changement de base, notre extension $\Omega$ de d\'epart provient d'une extension de $G$ par $\gm$ pour la topologie fppf. Or il est bien connu (voir \cite[expos\'e VIII, 3.3.1]{gro7} ou \cite[Theorem 1]{w}) qu'une telle extension est localement scind\'ee pour la topologie fppf, d'o\`u le r\'esultat.
\end{proof}

On en d\'eduit le r\'esultat suivant.

\begin{cor}
\label{isocano}
Soit $G$ un $X$-sch\'ema en groupes commutatif, fini localement libre. On dispose d'un isomorphisme
$$
\begin{CD}
H^1_{\kpl}(X,G^D) @>\sim>> \ext^1_{\kpl}(G,\gm) \\
\end{CD}
$$
canonique, et fonctoriel en $G$.
\end{cor}

\begin{proof}
On dispose d'une suite spectrale locale-globale pour les faisceaux Ext en topologie Kummer plate (comme dans n'importe quel topos annel\'e, cf. \cite[expos\'e V, proposition 6.1, 3)]{gro4t2}), laquelle donne naissance \`a une suite exacte
$$
0 \longrightarrow H^1_{\kpl}(X,\Hom(G,\gm)) \longrightarrow \ext^1_{\kpl}(G,\gm) \longrightarrow H^0(X,\Ext^1_{\kpl}(G,\gm))
$$
avec $\Hom(G,\gm)=G^D$. Le membre de droite \'etant nul d'apr\`es le th\'eor\`eme \ref{splitext}, on en d\'eduit le r\'esultat.
\end{proof}


\subsection{Structure galoisienne logarithmique}

Soit $G$ un $X$-sch\'ema en groupes commutatif, fini localement libre. Rappelons que l'on doit \`a Waterhouse (voir \cite[Theorem 5]{w}) la d\'efinition d'un morphisme, que nous appellerons morphisme de classes
$$
\begin{CD}
\pi:H^1_{\fppf}(X,G^D) @>>> \pic(G) \\
\end{CD}
$$
et qui mesure la structure galoisienne des $G^D$-torseurs pour la topologie fppf.

De fa\c{c}on purement formelle, il est facile de d\'efinir un homomorphisme semblable \`a celui de Waterhouse mesurant la structure galoisienne des $G^D$-torseurs pour la topologie Kummer log plate.

\begin{dfn}
\label{loghomo}
On appelle morphisme de log classes, et on note $\pi^{\rm log}$, le compos\'e des morphismes
$$
\begin{CD}
H^1_{\kpl}(X,G^D) @>\sim>> \ext^1_{\kpl}(G,\gm) @>>> H^1_{\kpl}(G,\gm) \\
\end{CD}
$$
dans lequel la premi\`ere fl\`eche est l'isomorphisme du corollaire \ref{isocano}, et la seconde est le morphisme naturel.
\end{dfn}

Rappelons tout d'abord un fait \'el\'ementaire. Si $\topo$ est une topologie de Grothendieck dans la cat\'egorie $(fs/X)$, alors nous avons une suite exacte
$$
\begin{CD}
0\longrightarrow \ext^1_{\mathcal{P},\topo}(G,\gm) @>>> \ext^1_{\topo}(G,\gm) @>>> H^1_{\topo}(G,\gm) \\
\end{CD}
$$
o\`u  $\ext^1_{\mathcal{P},\topo}$ d\'esigne le groupe des extensions calcul\'e dans la cat\'egorie des pr\'efaisceaux ab\'eliens sur $(fs/X)_{\topo}$. Le lemme ci-dessous est \'evident (la cat\'egorie des pr\'efaisceaux sur un site ne d\'ependant que de la cat\'egorie sous-jacente au site).

\begin{lem}
\label{extofpresheaves}
Le morphisme naturel
$$
\ext^1_{\mathcal{P},\fppf}(G,\gm)\longrightarrow \ext^1_{\mathcal{P},\kpl}(G,\gm)
$$
est un isomorphisme.
\end{lem}

En particulier, les noyaux de $\pi^{\rm log}$ et de $\pi$ sont isomorphes. C'est l'objet de la proposition suivante.

\begin{prop}
\label{kerpilog}
On dispose d'un diagramme commutatif
$$
\begin{CD}
H^1_{\fppf}(X,G^D) @>\pi>> \pic(G) \\
@VVV @VVV \\
H^1_{\kpl}(X,G^D) @>\pi^{\rm log}>> H^1_{\kpl}(G,\gm) \\
\end{CD}
$$
dans lequel les fl\`eches verticales sont injectives. En outre, le morphisme induit sur les noyaux
$$
\ker(\pi)\longrightarrow \ker(\pi^{\rm log})
$$
est un isomorphisme. Enfin, \'etant donn\'e un $G^D$-torseur log plat, celui-ci est un $G^D$-torseur fppf si et seulement si son image par $\pi^{\rm log}$ est un $\gm$-torseur fppf.
\end{prop}

\begin{proof}
L'assertion d'injectivit\'e d\'ecoule de la suite exacte \eqref{spec1}. Par d\'efinition de $\pi$ et $\pi^{\rm log}$, l'assertion sur les noyaux d\'ecoule du lemme \ref{extofpresheaves} ci-dessus. Le dernier point se r\'eduit \`a l'observation suivante : \'etant donn\'e une extension de $G$ par $\gm$ pour la topologie log plate, cette derni\`ere est une extension pour la topologie fppf si et seulement si son $\gm$-torseur sous-jacent est un torseur fppf.
\end{proof}


\subsection{Cas du groupe $\mu_n$}

Rappelons que la multiplication par un entier $n\geq 1$ sur le groupe $\gm$ est surjective en topologie log plate. Ceci induit une suite exacte courte
\begin{equation}
\label{gmkummer}
\begin{CD}
0 @>>> \gm(X)/n @>d>> H^1_{\kpl}(X,\mu_n) @>\s>> H^1_{\kpl}(X,\gm)[n] @>>> 0 \\
\end{CD}
\end{equation}

Dans le cas de la topologie fppf, cette suite exacte est le point de d\'epart d'une description fort agr\'eable des $\mu_n$-torseurs en termes de $\gm$-torseurs munis d'une trivialisation de leur puissance $n$-i\`eme. Avec les notations du paragraphe \ref{torseurscycliques}, ceci s'exprime par un isomorphisme canonique $H^1_{\fppf}(X,\mu_n)\simeq \Rac(X,\emptyset,n)$.

On peut transposer cela en topologie log plate. Consid\'erons la cat\'egorie des couples $(L,\tau)$ o\`u $L$ est un $\gm$-torseur pour la topologie log plate sur $X$, et $\tau$ est une section de $L^{\vee n}$, le symbole $\vee$ d\'esignant le produit contract\'e de $\gm$-torseurs. Munie du produit auquel on pense, c'est une cat\'egorie de Picard strictement commutative. Nous noterons $\Raclog(X,n)$ le groupe des classes d'isomorphie de ses objets.

\'Etant donn\'e une section $\alpha\in\gm(X)$, on note $\tau_{\alpha}$ la section de $\gm^{\vee n}$ obtenue en composant l'isomorphisme canonique $\gm^{\vee n}\simeq \gm$ avec la multiplication par $\alpha$. Ainsi le couple $(\gm,\tau_{\alpha})$ est un \'el\'ement de $\Raclog(X,n)$.
Nous d\'efinissons une suite
\begin{equation}
\label{clogkummer}
\begin{CD}
0 @>>> \gm(X)/n @>i>> \Raclog(X,n) @>o>> H^1_{\kpl}(X,\gm)[n] @>>> 0 \\
\end{CD}
\end{equation}
o\`u $i$ est l'application $\alpha\mapsto(\gm,\tau_{\alpha})$, et $o$ est l'oubli $(L,\tau)\mapsto L$. On v\'erifie ais\'ement que $i$ et $o$ sont des morphismes, et que cette suite est exacte.

\begin{prop}
\label{muntorsdesc}
On dispose d'un isomorphisme canonique
$$
\begin{CD}
\omega:\Raclog(X,n) @>\sim>> H^1_{\kpl}(X,\mu_n) \\
\end{CD}
$$
qui s'inscrit dans un isomorphisme entre les suites exactes \eqref{gmkummer} et \eqref{clogkummer}, les fl\`eches sur les noyaux et les conoyaux \'etant les identit\'es.
\end{prop}

\begin{proof}
Soit $(L,\tau)$ un \'el\'ement de $\Raclog(X,n)$. Alors $L$ est de $n$-torsion, donc provient d'un $\mu_n$-torseur en vertu de la suite \eqref{gmkummer}. Les $\mu_n$-torseurs \'etant repr\'esentables (th\'eor\`eme \ref{representable}), on en d\'eduit que $L$ est repr\'esentable. Consid\'erons le log sch\'ema en groupes $\Omega(L,\tau)$ sur $X$ d\'efini par
$$
\Omega(L,\tau)=\gm\sqcup L\sqcup L^{\vee 2}\sqcup \dots \sqcup L^{\vee (n-1)}
$$
le symbole $\sqcup$ d\'esignant la r\'eunion disjointe dans la cat\'egorie des log sch\'emas. La loi de multiplication sur $\Omega(L,\tau)$ est induite par le produit contract\'e et par la section $\tau$, laquelle fournit un isomorphisme $L^{\vee n}\simeq \gm$. De plus, on dispose d'une suite exacte
$$
\begin{CD}
0 @>>> \gm @>>> \Omega(L,\tau) @>>> (\mathbb{Z}/n\mathbb{Z})_X @>>> 0 \\
\end{CD}
$$
o\`u les morphismes sont d\'efinis de fa\c{c}on \'evidente, le log sch\'ema sous-jacent au groupe $(\mathbb{Z}/n\mathbb{Z})_X$ \'etant la r\'eunion disjointe de $n$ copies de $X$. Cette suite exacte d\'efinit un \'el\'ement de $\ext^1_{\kpl}((\mathbb{Z}/n\mathbb{Z})_X,\gm)$, lequel groupe est isomorphe \`a $H^1_{\kpl}(X,\mu_n)$ en vertu du corollaire \ref{isocano}. Nous obtenons ainsi une application $\omega:\Raclog(X,n)\rightarrow H^1_{\kpl}(X,\mu_n)$, qui s'av\`ere \^etre un morphisme de groupes (v\'erification formelle). Pour montrer qu'il s'agit d'un isomorphisme, il suffit de v\'erifier que $\omega$ d\'efinit un morphisme entre les suites \eqref{gmkummer} et \eqref{clogkummer}, en prenant les identit\'es sur les noyaux et conoyaux. Autrement dit, il suffit de v\'erifier que $\omega\circ i=d$ et que $\s\circ\omega=o$, ce qui est un calcul imm\'ediat.
\end{proof}

Une fois ce r\'esultat \'etabli, on constate que le morphisme de log classes pour le groupe $G^D=\mu_n$ transporte les m\^emes informations que le morphisme naturel $\s$ ci-dessus, ce qui g\'en\'eralise le r\'esultat analogue en topologie fppf (voir \cite[prop. 3.1]{gil4}).

\begin{prop}
\label{commebx}
Soit $n\geq 1$ un entier naturel. Alors le morphisme de log classes
$$
\begin{CD}
\pi^{\rm log}:H^1_{\kpl}(X,\mu_n) @>>> H^1_{\kpl}((\mathbb{Z}/n\mathbb{Z})_X,\gm)\simeq H^1_{\kpl}(X,\gm)^n \\
\end{CD}
$$
est l'application qui \`a un torseur $T$ associe le $n$-uplet $(0,\s(T),\s(T)^2,\dots,\s(T)^{n-1})$. En particulier, l'image de $\pi^{\rm log}$ est isomorphe au groupe $H^1_{\kpl}(X,\gm)[n]$.
\end{prop}

\begin{proof}
Soit $T\rightarrow X$ un $\mu_n$-torseur pour la topologie log plate. D'apr\`es la preuve de la proposition \ref{muntorsdesc}, l'image de $T$ par l'isomorphisme du corollaire \ref{isocano}
$$
\begin{CD}
H^1_{\kpl}(X,\mu_n) @>\sim>> \ext^1_{\kpl}((\mathbb{Z}/n\mathbb{Z})_X,\gm) \\
\end{CD}
$$
est une extension de la forme $\Omega(L,\tau)$, et $L=\s(T)$ d'apr\`es la proposition \ref{muntorsdesc}. D'apr\`es la description de $\Omega(L,\tau)$, le $\gm$-torseur sous-jacent \`a cette derni\`ere est le $n$-uplet
$$
(\gm,L,L^{\vee 2},\dots,L^{\vee (n-1)})
$$
Enfin, d'apr\`es la d\'efinition \ref{loghomo}, ce $\gm$-torseur n'est autre que $\pi^{\rm log}(T)$. On en d\'eduit que l'image de $\pi^{\rm log}$ est isomorphe \`a celle de $\s$, c'est-\`a-dire \`a $H^1_{\kpl}(X,\gm)[n]$.
\end{proof}

Dans l'exemple ci-dessous, nous regardons la structure galoisienne logarithmique des rev\^etements cycliques uniformes.

\begin{exm}
\label{logstrexm}
Soit $X$ un sch\'ema n\oe{}th\'erien r\'egulier, muni de la log structure associ\'ee \`a un ouvert $U$ dont le compl\'ementaire $D$ est un diviseur \`a croisements normaux sur $X$. D'apr\`es la proposition \ref{commebx}, pour calculer la structure galoisienne logarithmique des $\mu_n$-torseurs, il suffit d'expliciter l'application $\s$.

Notre point de d\'epart est la donn\'ee d'un $\mu_n$-torseur log plat (\emph{i.e.} d'un rev\^etement cyclique uniforme) $T\rightarrow X$. D'apr\`es le corollaire \ref{logcyclor}, on peut le d\'ecrire sous la forme $\Sa(\mathcal{L},\phi)$ o\`u $(\mathcal{L},\phi)$ est un \'el\'ement de $\Rac(X,D,n)$. Soit $f$ une section m\'eromorphe de $\mathcal{L}$, et soit $g$ une section m\'eromorphe de l'id\'eal image de $\phi$ (lequel d\'efinit le diviseur de branchement $\Delta_{T/X}$ de $T\rightarrow X$).

Rappelons que l'on note $\DivRat(X,D)$ le groupe des diviseurs (sur $X$) \`a coefficients rationnels au-dessus de $D$ (voir \cite[def 3.1.1]{gil5}). On peut associer au torseur $T$ le diviseur \`a coefficients rationnels suivant
$$
L(T):=-\diviseur(f)+\frac{1}{n}\diviseur(g)
$$
On constate alors que la classe $[L(T)]$ de $L(T)$ modulo le groupe $\Divp(X)$ des diviseurs principaux est \'egale \`a $\s(T)$, le groupe $H^1_{\kpl}(X,\gm)$  s'identifiant canoniquement au quotient $\DivRat(X,D)/\Divp(X)$ (voir \cite[th\'eor\`eme 3.1.3]{gil5}). En fait, on peut m\^eme dire plus : le torseur $T$ s'identifie (via l'isomorphisme du th\'eor\`eme \ref{muntorsdesc}) au couple $([L(T)],\tau)$ o\`u $\tau$ est la trivialisation de $n[L(T)]$ fournie par l'isomorphisme $\phi:\mathcal{L}^{\otimes n}\simeq \mathcal{O}_X(\Delta_{T/X})$.
\end{exm}

Dans le cas particulier d'un trait muni de sa log structure canonique, le morphisme de log classes mesure la ramification des extensions cycliques.

\begin{exm}
\label{logramiftrait}
Soit $S$ un trait, de point g\'en\'erique $\Spec(K)$, muni de sa log structure canonique. Alors on constate que le morphisme $\s$
$$
K^*/(K^*)^n=H^1_{\kpl}(S,\mu_n)\longrightarrow \mathbb{Q}/\mathbb{Z}=H^1_{\kpl}(S,\gm)
$$
envoie $z\in K^*$ sur $v(z)/n$, o\`u $v$ est la valuation de $S$. Par cons\'equent, l'ordre de $\s(z)$ dans $\mathbb{Q}/\mathbb{Z}$ n'est autre que l'indice de ramification de la $K$-alg\`ebre $K(\sqrt[n]{z})$.
\end{exm}


\subsection{Structure galoisienne classique}
\label{mapclass}

Comme pr\'ec\'edemment, on se donne un $X$-sch\'ema en groupes commutatif, fini localement libre $G$.
Sous les hypoth\`eses du corollaire \ref{CHtame}, on peut d\'efinir une notion de structure galoisienne pour les $G$-torseurs log plats.

\begin{dfn}
\label{mapcl}
On suppose que $X$ admet, localement pour la topologie \'etale sur $X$, une carte model\'ee sur un mono\"ide de la forme $\mathbb{N}^r$, et que l'ouvert de trivialit\'e est dense dans $X$. Alors on d\'efinit une application ensembliste
$$
\begin{CD}
\cl:H^1_{\kpl}(X,G^D) @>>> \pic(G) \\
\end{CD}
$$
qui envoie un $G^D$-torseur $Y\rightarrow X$ sur la classe de $\mathcal{O}_Y\otimes_{\mathcal{O}_G} (\mathcal{O}_{G^D})^{-1}$ dans le groupe $\pic(G)$. Il serait l\'egitime d'appeler cette application \emph{structure galoisienne classique}.
\end{dfn}

Quand on restreint l'application $\cl$ au sous-groupe $H^1_{\fppf}(X,G^D)\subseteq H^1_{\kpl}(X,G^D)$, on retrouve le morphisme de classes $\pi$ d\'efini par Waterhouse. Notons cependant que $\cl$ n'est pas un morphisme de groupes en g\'en\'eral, comme le montre l'exemple qui suit.

\begin{exm}
\label{CDNconstant}
Soient $K$ un corps de nombres, $X=\Spec(\mathcal{O}_K)$ le spectre de l'anneau des entiers de $K$, $D$ un ensemble fini de points ferm\'es de $X$. On munit $X$ de la log structure d\'efinie par l'ouvert compl\'ementaire de $D$. Soit $\Gamma$ un  groupe ab\'elien fini, nous identifions $H^1_{\text{\rm \'et}}(K,\Gamma)$ \`a l'ensemble des (classes d'isomorphie de) $K$-alg\`ebres galoisiennes de groupe $\Gamma$. Alors
$$
H^1_{\ket}(X,\Gamma)\subseteq H^1_{\et}(K,\Gamma)
$$
s'identifie au sous-groupe des $K$-alg\`ebres qui sont mod\'er\'ement ramifi\'ees en les points de $D$ et non ramifi\'ees partout ailleurs, tandis que $H^1_{\et}(X,\Gamma)$ est constitu\'e des $K$-alg\`ebres partout non ramifi\'ees. De plus, le sch\'ema sous-jacent \`a un $\Gamma$-torseur $Y\rightarrow X$ est le spectre de l'anneau des entiers de la $K$-alg\`ebre qui lui correspond via l'inclusion ci-dessus.

Enfin, $\cl$ est l'application bien connue
$$
\cl:H^1_{\ket}(X,\Gamma)\longrightarrow  \pic(\Gamma^D)\simeq \Cl(\mathcal{O}_K[\Gamma])
$$
qui envoie une $K$-alg\`ebre de groupe $\Gamma$ (mod\'er\'ement ramifi\'ee en les points de $D$) sur la classe de son anneau d'entiers dans le groupe des classes localement libres $\Cl(\mathcal{O}_K[\Gamma])$.

On sait que dans ce cadre, $\cl$ n'est pas en g\'en\'eral un morphisme de groupes (voir \cite[Remark 5.2]{childs}). Par contre, l'image de $\cl$ est un sous-groupe de $\Cl(\mathcal{O}_K[\Gamma])$ d'apr\`es un c\'el\`ebre r\'esultat de McCulloh \cite{mcc}.
\end{exm}

\begin{exm}
\label{quelonveut}
Avec les hypoth\`eses et notations de l'exemple pr\'ec\'edent, l'ouvert de trivialit\'e de $X$ est $U=\Spec(\mathcal{O}_{K,D})$, o\`u $\mathcal{O}_{K,D}$ d\'esigne l'anneau des $D$-entiers de $K$. Nous noterons $X^0$ l'ensemble des points ferm\'es de $X$ (\emph{i.e.} l'ensemble des id\'eaux premiers non nuls de $\mathcal{O}_K$). Pour tout $\mathfrak{p}\in X^0$, nous noterons $v_\mathfrak{p}$ la valuation associ\'ee.
Soit  $n\geq 1$ un entier naturel, alors le groupe $H^1_{\kpl}(X,\mu_n)$ est isomorphe au groupe $H^1_{\fppf}(U,\mu_n)$ en vertu de la proposition \ref{dejafait}. Ce dernier, vu comme sous-groupe de $H^1_{\fppf}(\Spec(K),\mu_n)=K^*/(K^*)^n$, admet la description suivante :
$$
H^1_{\kpl}(X,\mu_n)=\{z\in K^*/(K^*)^n\,|\, \forall \mathfrak{p}\in X^0\backslash D, v_\mathfrak{p}(z)\equiv 0\; (\textrm{mod } n)\}
$$
Choisissons un \'el\'ement $T$ de ce groupe, repr\'esent\'e par un $z\in K^*$. Quitte \`a multiplier $z$ par une puissance $n$-i\`eme, on peut supposer que $z$ appartient \`a $\mathcal{O}_K$. Nous allons calculer $\cl(T)$ en fonction de $z$. On peut \'ecrire $T$ sous la forme $(I,\phi)$ o\`u $I$ est un id\'eal fractionnaire de $K$ et $\phi:I^n\mapsto \mathcal{O}_K$ est une injection. Soit $J$ l'id\'eal image de $\phi$, de sorte que $J$ d\'efinit le diviseur de branchement de $T$. En consid\'erant ce qui se passe sur la fibre g\'en\'erique, on obtient la relation $I^{-n}J=(z)$ dans le groupe des id\'eaux fractionnaires.

Pour tout point $\mathfrak{p}\in X^0$, on note
$$
v_\mathfrak{p}(z)=nq_{\mathfrak{p}}(z)+r_{\mathfrak{p}}(z)
$$
la division euclidienne de $v_\mathfrak{p}(z)$ par $n$. Notons que, pour tout $\mathfrak{p}\not\in D$, le reste $r_{\mathfrak{p}}(z)$ est nul. En traduisant ce qu'on a vu dans l'exemple \ref{logstrexm}, on peut alors \'ecrire
$$
\s(T)=[\frac{1}{n} \diviseur(z)]=\sum_{\mathfrak{p}\in X^0} q_{\mathfrak{p}}(z).[\mathfrak{p}]+[\frac{r_{\mathfrak{p}}(z)}{n}\mathfrak{p}]=-[I]+[\frac{1}{n} J]
$$
En oubliant la partie \`a coefficients rationnels, et en changeant de signe, on obtient une classe $c(z)$ dans le groupe $\pic(X)$,
$$
c(z)=-\sum_{\mathfrak{p}\in X^0} q_{\mathfrak{p}}(z).[\mathfrak{p}]
$$
et $\cl(T)$ est l'\'el\'ement $(0,c(z),c(z)^2,\dots,c(z)^{n-1})$ dans le groupe $\pic((\mathbb{Z}/n\mathbb{Z})_X)\simeq \pic(X)^n$. Notons que $c(z)$ n'est pas forc\'ement \'egal \`a la classe de $I$.
Enfin, on constate (avec les notations du paragraphe \ref{torseurscycliques}) que
$$
\nu_n(T)=\sum_{\mathfrak{p}\in D} \frac{v_\mathfrak{p}(z)}{n}.\mathfrak{p}
$$
dans le groupe $\oplus_{\mathfrak{p}\in D} (\frac{1}{n}\mathbb{Z}/\mathbb{Z}).\mathfrak{p}$.
\end{exm}


\subsection{Vari\'et\'es ab\'eliennes et invariants de classes}
\label{dernier}

Dans ce paragraphe, nous reprenons les notations et hypoth\`eses du paragraphe \ref{hundeux} : $S$ est un sch\'ema de Dedekind  connexe de point g\'en\'erique $\eta=\Spec(K)$, $\A$ (resp. $\A^t$) est le mod\`ele de N\'eron d'une $K$-vari\'et\'e ab\'elienne (resp. de sa vari\'et\'e duale), et $U\subseteq S$ est l'ouvert de bonne r\'eduction de $\A$.

On munit $S$ de la log structure d\'efinie par $U$. Le log sch\'ema ainsi obtenu satisfait les hypoth\`eses des propositions \ref{loclibre} et \ref{injectrestr}, ainsi que celles du corollaire \ref{CHtame}.

D'apr\`es \cite[Corollaire 4.1.8]{gil5}, il existe une unique biextension $W^{\rm log}$ de $(\A,\A^t)$ par $\gm$ pour la topologie log plate qui prolonge la biextension de Weil $W_U$.

La d\'emarche que nous adoptons \`a pr\'esent est semblable \`a celle utilis\'ee dans \cite{gil1}. Nous commen\c{c}ons par introduire un petit site Kummer plat.

\begin{dfn}
On appelle petit site Kummer plat sur $S$ la cat\'egorie des $S$-log sch\'emas fins et satur\'es dont le morphisme structural est log plat et kumm\'erien, munie d'une structure de site pour la topologie Kummer log plate.
\end{dfn}

A pr\'esent, tous les faisceaux consid\'er\'es seront des faisceaux ab\'eliens sur le petit site Kummer plat de $S$.

\begin{lem}
\label{nulhom}
Le faisceau $\Hom(\A,\gm)$ sur le petit site Kummer plat est nul.
\end{lem}

\begin{proof}
Soit $f:Z\rightarrow S$ un morphisme log plat kumm\'erien, nous devons montrer que
$$
\homr_Z(\A_Z,\gmy)=0
$$
Les log structures consid\'er\'ees \'etant images r\'eciproques de la log structure de $Z$, la donn\'ee d'un morphisme de log sch\'emas $\A_Z\rightarrow\gmy$ \'equivaut \`a la donn\'ee d'un morphisme sur les sch\'emas sous-jacents.

D'autre part, la question est locale pour la topologie fppf sur $S$ et $Z$. D'apr\`es le lemme \ref{niziolstyle} on sait que, localement pour la topologie fppf sur $S$ et $Z$, le morphisme $f:Z\rightarrow S$ admet une carte qui est un morphisme kumm\'erien de mono\"ides fs. Ainsi on se ram\`ene au cas o\`u $f=g\circ h$ avec $g:T\rightarrow S$ un torseur kumm\'erien standard, et $h:Z\rightarrow T$ un morphisme strict, plat sur les sch\'emas sous-jacents. D'apr\`es la proposition \ref{loclibre}, $g$ est plat sur les sch\'emas sous-jacents, donc $f$ l'est aussi. On en d\'eduit que l'ouvert $V:=f^{-1}(U)$ est dense dans $Z$. On peut alors \'ecrire
$$
\homr_Z(\A_Z,\gmy)\subseteq \homr_V(\A_V,\gmv)=0
$$
l'inclusion d\'ecoulant de la densit\'e de $V$ dans $Z$, et de la platitude des sch\'emas $\A_Z$ et $\gmy$ sur $Z$. La nullit\'e du second membre est un r\'esultat bien connu puisque $\A_V$ est un $V$-sch\'ema ab\'elien.
\end{proof}

Les axiomes de base sont r\'eunis pour que la construction de \cite{gil1} fonctionne sur le petit site Kummer plat.

Soit $G\subseteq \A$ un sous-groupe fini et plat de $\A$. On quotiente $\A$ par $G$ dans la cat\'egorie des faisceaux sur le petit site Kummer plat, ce qui donne lieu \`a une suite exacte
$$
\begin{CD}
0 @>>> G @>>> \A @>\phi>> B @>>> 0 \\
\end{CD}
$$
o\`u $B$ est un faisceau, \emph{a priori} non repr\'esentable.

On applique ensuite le foncteur $\Hom(-,\gm)$ \`a cette suite. Les faisceaux $\Hom(\A,\gm)$ et $\Ext^1_{\kpl}(G,\gm)$ \'etant nuls (lemme \ref{nulhom} et th\'eor\`eme \ref{splitext}), on obtient une suite exacte courte
$$
\begin{CD}
0 @>>> G^D @>>> \Ext^1_{\kpl}(B,\gm) @>\phi^*>> \Ext^1_{\kpl}(\A,\gm) @>>> 0 \\
\end{CD}
$$
d'o\`u un morphisme cobord
$$
\begin{CD}
\delta^{\rm log}:\ext^1_{\kpl}(\A,\gm) @>>> H^1_{\kpl}(S,G^D) \\
\end{CD}
$$
et par composition avec l'isomorphisme $\gamma^{\rm log}:\A^t(S)\rightarrow \ext^1_{\kpl}(\A,\gm)$ (associ\'e \`a la biextension $W^{\rm log}$), un morphisme
$$
\begin{CD}
\Delta:\A^t(S) @>>> H^1_{\kpl}(S,G^D) \\
\end{CD}
$$
On en d\'eduit le point $(i)$ du th\'eor\`eme \ref{thintro1}. On peut par la suite mesurer la structure galoisienne des torseurs (pour la topologie Kummer plate) obtenus gr\^ace \`a ce morphisme, d'o\`u le th\'eor\`eme \ref{thintro2} de l'introduction.

En reprenant (sans les modifier) les arguments de \cite[lemme 3.2 et remarque 3.3]{gil1}, on montre le r\'esultat suivant.

\begin{lem}
\label{geomdesk}
Soit $i_G:G\rightarrow \A$ l'inclusion canonique, et soit $y$ un point de $\A^t(S)$. Alors le pull-back $(i_G\times y)^*(W^{\rm log})$ est une extension de $G$ par $\gm$ (pour la topologie log plate). Cette extension n'est autre que l'image du $G^D$-torseur $\Delta(y)$ par l'isomorphisme du corollaire \ref{isocano}.
\end{lem}

\begin{rmq}
Pour construire $\Delta$, nous avons utilis\'e la m\^eme d\'emarche que dans \cite{gil1}, \`a savoir expliciter une suite exacte permettant de construire les $G^D$-torseurs. Au vu du lemme \ref{geomdesk}, nous aurions pu d\'efinir $\Delta(y)$ comme \'etant le $G^D$-torseur correspondant \`a l'extension $(i_G\times y)^*(W^{\rm log})$ par l'isomorphisme du corollaire \ref{isocano}.
Ceci permet de montrer le $(i)$ du th\'eor\`eme \ref{thintro1} sans utiliser le petit site Kummer plat, les seuls outils n\'ecessaires \'etant la biextension $W^{\rm log}$ et l'isomorphisme du corollaire \ref{isocano}.
\end{rmq}

Il nous reste \`a \'etablir le point $(ii)$ du th\'eor\`eme \ref{thintro1}. En vertu de la proposition \ref{kerpilog}, nous pouvons dire que $\Delta(y)$ est un $G^D$-torseur fppf si et seulement si $\psi^{\rm log}(y)$ est un $\gm$-torseur fppf.
Soit  $t(W^{\rm log})$ le $\gm$-torseur sur $\A\times_S\A^t$ sous-jacent \`a la biextension $W^{\rm log}$. Gr\^ace au lemme \ref{geomdesk} (et en reprenant les arguments de \cite[remarque 3.3]{gil1}) nous avons, pour tout $y\in\A^t(S)$, l'\'egalit\'e suivante
\begin{equation}
\label{classinvariant}
\psi^{\rm log}(y)=(i_G\times y)^*(t(W^{\rm log}))
\end{equation}
dans le groupe $H^1_{\kpl}(G,\gm)$.

Le probl\`eme \'etant de nature locale pour la topologie \'etale, nous pouvons supposer que $S$ est un trait strictement hens\'elien. Nous noterons $i:s\rightarrow S$ l'inclusion du point ferm\'e $s=\Spec(k)$, o\`u $k$ est un corps s\'eparablement clos. Soit $\Phi=\A_s/\A_s^0$, et soit $i_*\Gamma$ l'image de $G$ par le morphisme canonique $\A\rightarrow i_*\Phi$. Alors $\Gamma$ est un $k$-sch\'ema en groupes fini \'etale, donc constant, et le faisceau $i_*\Gamma$ est repr\'esentable par un $S$-sch\'ema en groupes constant tronqu\'e. En restreignant \`a $G$ le morphisme $\A\rightarrow i_*\Phi$, on obtient un morphisme (surjectif) de $S$-sch\'emas $G\rightarrow i_*\Gamma$. Nous munissons ces deux sch\'emas de la log structure image r\'eciproque de celle de $S$. On constate alors, en se servant de l'\'egalit\'e \eqref{classinvariant} et d'un argument de faisceautisation semblable \`a celui de \cite[proposition 4.2.3]{gil5}, que le diagramme suivant commute
$$
\begin{CD}
\A^t(S) @>\psi^{\rm log}>> H^1_{\kpl}(G,\gm) @>\nu >> H^0(G,\mathbb{Q}/\mathbb{Z} \otimes \overline{M}_G^{\rm gp}) \\
@VVV @. @AAeA \\
\Phi'(k) @>m>> \bigoplus_{\gamma\in \Gamma(k)} (\mathbb{Q}/\mathbb{Z}).\gamma @>\sim>> H^0(i_*\Gamma,\mathbb{Q}/\mathbb{Z} \otimes \overline{M}_{i_*\Gamma}^{\rm gp}) \\
\end{CD}
$$
dans lequel les fl\`eches sont les suivantes : $\nu$ se d\'eduit de la suite spectrale \eqref{spec1} et du th\'eor\`eme \ref{epsilonkato}, la fl\`eche verticale de gauche est obtenue par r\'eduction \`a la fibre sp\'eciale, la fl\`eche $m$ est d\'efinie par
$$
m(\zeta)=\sum_{\gamma\in \Gamma(k)} <\gamma,\zeta>^{\rm mono}.\gamma
$$
(o\`u $<\gamma,\zeta>^{\rm mono}$ est l'accouplement de monodromie entre les composantes $\gamma$ et $\zeta$), et la fl\`eche $e$ d\'ecoule de l'isomorphisme naturel $q^*M_{i_*\Gamma}\simeq M_G$ correspondant au morphisme strict de log sch\'emas $q:G\rightarrow i_*\Gamma$. On v\'erifie ais\'ement que $e$ est injective.

On sait que le noyau de $\nu$ est l'ensemble des $\gm$-torseurs fppf. On d\'eduit alors du diagramme ci-dessus (et de l'injectivit\'e de $e$) le fait suivant : \'etant donn\'e $y\in \A^t(S)$, $\psi^{\rm log}(y)$ est un $\gm$-torseur fppf si et seulement si $m(\zeta_y)=0$, o\`u $\zeta_y$ est la composante dans laquelle $y$ se r\'eduit. En vertu de ce qui pr\'ec\`ede, nous venons donc de montrer le $(ii)$ du th\'eor\`eme \ref{thintro1}.

Supposons encore une fois que $S$ soit un trait. Nous notons $(-,-)^{\rm mono}$ l'accouplement $\A(S)\times \A^ t(S)\rightarrow \mathbb{Q}/\mathbb{Z}$ obtenu en composant l'accouplement de monodromie avec le morphisme de r\'eduction \`a la fibre sp\'eciale. Le corollaire ci-dessous permet de relier la ramification des extensions obtenues en divisant des points aux valeurs de l'accouplement de monodromie.

\begin{cor}
Supposons que $S$ soit un trait et que $G$ soit cyclique constant engendr\'e par un certain $x\in \A(S)$. Alors, pour tout $y\in \A^t(S)$, l'indice de ramification de l'alg\`ebre de $\Delta(y)$ est l'ordre de l'accouplement de monodromie $(x,y)^{\rm mono}$ dans $\mathbb{Q}/\mathbb{Z}$.
\end{cor}

\begin{proof}
D\'ecoule du lemme \ref{geomdesk}, de l'exemple \ref{logramiftrait} et de \cite[Proposition 4.2.3]{gil5}.
\end{proof}



\vskip 1cm

Jean Gillibert
\smallskip

Institut de Math\'ematiques de Bordeaux

Universit\'e Bordeaux 1

351, cours de la Lib\'eration

33405 Talence Cedex

France

\medskip

\texttt{jean.gillibert@math.u-bordeaux1.fr}

\end{document}